\documentclass[11pt,twoside,a4paper]{article}
\usepackage{color,amscd,amsmath,amssymb,amsthm,latexsym,stmaryrd,authblk,mathabx,shuffle,mathrsfs,hyperref,tikz,
tkz-tab} 
\usepackage[latin1]{inputenc}
\usepackage[T1]{fontenc}   
\usepackage[english]{babel}
\providecommand{\keywords}[1]{\textbf{\textit{Keywords.}} #1}
\providecommand{\AMSclass}[1]{\textbf{\textit{AMS classification.}} #1}

\input{xy}
\xyoption{all}

\title{Hopf algebraic structures on hypergraphs and multi-complexes}
\date{}
\author{Lo\"ic Foissy}
\affil{\small{Univ. Littoral Côte d'Opale, UR 2597
LMPA, Laboratoire de Mathématiques Pures et Appliquées Joseph Liouville
F-62100 Calais, France}.\\ Email: \texttt{foissy@univ-littoral.fr}}

\newcommand{\tun}{\begin{tikzpicture}[line cap=round,line join=round,>=triangle 45,x=0.5cm,y=0.5cm]
\clip(-0.2,-0.1) rectangle (0.2,0.2);
\begin{scriptsize}
\draw [fill=black] (0.,0.) circle (1pt);
\end{scriptsize}
\end{tikzpicture}}

\newcommand{\tdeux}{\begin{tikzpicture}[line cap=round,line join=round,>=triangle 45,x=0.5cm,y=0.5cm]
\clip(-.2,-.1) rectangle (0.2,0.7);
\draw [line width=.5pt] (0.,0.5)-- (0.,0.);
\begin{scriptsize}
\draw [fill=black] (0.,0.) circle (1pt);
\draw [fill=black] (0.,0.5) circle (1pt);
\end{scriptsize}
\end{tikzpicture}}

\theoremstyle{plain}
\newtheorem{theo}{Theorem}[section]
\newtheorem{lemma}[theo]{Lemma}
\newtheorem{cor}[theo]{Corollary}
\newtheorem{prop}[theo]{Proposition}
\newtheorem{defi}[theo]{Definition}

\theoremstyle{remark}
\newtheorem{remark}{Remark}[section]
\newtheorem{notation}{Notations}[section]
\newtheorem{example}{Example}[section]

\newcommand{\K}{\mathbb{K}}
\newcommand{\N}{\mathbb{N}}
\newcommand{\Z}{\mathbb{Z}}
\newcommand{\calG}{\mathscr{G}}
\newcommand{\bfG}{\mathbf{G}}

\newcommand{\id}{\mathrm{Id}}

\newcommand{\bfP}{\mathbf{P}}
\newcommand{\bfQ}{\mathbf{Q}}

\newcommand{\eq}{\mathcal{E}}
\newcommand{\cl}{\mathrm{cl}}

\newcommand{\sym}{\mathfrak{S}}
\newcommand{\Char}{\mathrm{Char}}

\newcommand{\tdelta}{\tilde{\Delta}}

\newcommand{\vect}{\mathrm{Vect}}
\renewcommand{\ker}{\mathrm{Ker}}
\newcommand{\prim}{\mathrm{Prim}}

\newcommand{\cc}{\mathrm{cc}}

\newcommand{\calF}{\mathcal{F}}

\newcommand{\Ehr}{\mathrm{Ehr}}
\newcommand{\calH}{\mathcal{H}}
\newcommand{\bfH}{\mathbf{H}}

\newcommand{\coinv}{\mathrm{coInv}}
\newcommand{\calA}{\mathcal{A}}
\newcommand{\calB}{\mathcal{B}}
\newcommand{\calP}{\mathcal{P}}
\newcommand{\calN}{\mathcal{N}}
\newcommand{\supp}{\mathrm{supp}}
\newcommand{\bfMC}{\mathbf{MC}}
\newcommand{\calMC}{\mathcal{MC}}

\begin{document}

\maketitle

\begin{abstract}
Using the formalism of species and twisted objects, we introduce two structures of cointeracting bialgebras on hypergraphs,
induced by two notions of induced sub-hypergraphs. We study the associated unique morphisms of cointeracting bialgebras
from hypergraphs to the polynomial algebra in one indeterminate: in the first case, this gives the chromatic polynomial of a graph attached to the considered hypergraph.
In the second case, we obtained Helgason's notion of chromatic polynomial of a hypergraph.
We obtain Hopf-algebraic proves of results about the values of this chromatic polynomial in $-1$ or
about its coefficients, with the help of the action of a monoid of characters. 
This allows to give multiplicity-free formulas for the antipodes of these objects, using various notions of
acyclic orientations of hypergraphs. 

Mixing the two notions of induced sub-hypergraphs, we obtain a third Hopf algebra, firstly described by Aguiar and Ardila.
We obtain negative results on the existence of a second coproduct making it a cointeracting bialgebra. 
Anyway, it is still possible to obtain a polynomial invariant from this structure, which is the chomatic polynomial
described by Aval, Kharagbossian and Tanasa. 

We finally study Iovanov and Jaiung's  Hopf algebra of multi-complexes, making it a cointeracting bialgebra
which has for quotient one of the preceding cointeracting bialgebras of hypergraphs.
\end{abstract}

\\

\AMSclass{16T05 16T30 05C15 05C25}
\tableofcontents

\section*{Introduction}

Hypergraphs (a name due to Claude Berge in the sixties) are generalisations of graphs, 
where edges can contain an arbitrary number of vertices. A lot of classical notions on graphs can be extended to hypergraphs:
sub-hypergraphs, colourings, orientations, and so on, see for example \cite{Berge89,Bretto2013,Butjas2015,Helgason74,Tomescu98,Voloshin2009}.
We are here interested in Hopf-algebraic aspects of hypergraphs, with applications to colouring and orientations in the spirit of the results obtained in \cite{Foissy36} for graphs.\\

We firstly construct four graded and connected Hopf algebras of hypergraphs, all based on the space $\calF[\bfH]$ generated by
the isoclasses of hypergraphs. For convenience, we choose to work in the framework
of species \cite{Joyal1981,Joyal1986}, instead of "classical" Hopf algebras, which are obtained by application of the bosonic
Fock functor defined by Aguiar and Mahajan \cite{Aguiar2010}. Note that coloured versions of these Hopf algebras can also be obtained by the application of
coloured Fock functors \cite{Foissy41}. For all these Hopf algebras, the product is the disjoint union of hypergraphs.
The coproducts are based on two different notions of induced sub-hypergraphs: if $G$ is a hypergraph and 
$I$ is a subset of the set of vertices of $G$, the edges  of the induced sub-hypergraph $G_{\mid_\subset I}$ are the edges of $G$ included in $I$,
 whereas the edges of the induced sub-hypergraph $G_{\mid_\cap I}$ are the intersections of 
edges of $G$ with $I$. This allows to define four coproducts: for $(\leftthreetimes,\rightthreetimes)\in \{\subset,\cap\}^2$,
the coproduct $\Delta^{(\leftthreetimes,\rightthreetimes)}$ is given on any hypergraph $G$ by
\[\Delta^{(\leftthreetimes,\rightthreetimes)}(G)=\sum_{I\subset V(G)} G_{\mid_\leftthreetimes I}\otimes
G_{\mid_\rightthreetimes V(G)\setminus I},\]
where $V(G)$ is the set of vertices of $G$. For example, denoting by $T_n$ the hypergraph with $n$ vertices and a unique 
edge containing all its vertices,
\begin{align*}
\Delta^{(\subset,\subset)}(T_n)&=T_n\otimes 1+1\otimes T_n+\sum_{k=1}^{n-1}\binom{n}{k} T_1^k\otimes T_1^{n-k},\\
\Delta^{(\cap,\cap)}(T_n)&=T_n\otimes 1+1\otimes T_n+\sum_{k=1}^{n-1}\binom{n}{k} T_k\otimes T_{n-k},\\
\Delta^{(\cap,\subset)}(T_n)&=T_n\otimes 1+1\otimes T_n+\sum_{k=1}^{n-1}\binom{n}{k} T_k\otimes T_1^{n-k},\\
\Delta^{(\subset,\cap)}(T_n)&=T_n\otimes 1+1\otimes T_n+\sum_{k=1}^{n-1}\binom{n}{k} T_1^k\otimes T_{n-k}.
\end{align*}
The two coproducts $\Delta^{(\subset,\subset)}$ and $\Delta^{(\cap,\cap)}$ are cocommutative,
whereas $\Delta^{(\cap,\subset)}$ and $\Delta^{(\subset,\cap)}$ are opposite one from the other. 
All these Hopf algebras contain the Hopf algebra of graphs of \cite{Foissy36}. 
Up to a quotient, $\Delta^{(\subset,\subset)}$ and $\Delta^{(\subset,\cap)}$ appear in the recent paper \cite{Ebrahimi-Fard2022},
under the notations $\Delta$ and $\Delta'$. 
The coproduct $\Delta^{(\subset,\cap)}$ is introduced in \cite{Aguiar2017} and studied in \cite{Aval2019}.

We then define two other coproducts $\delta^{(\subset)}$ and $\delta^{(\cap)}$ of contractions and extractions.
This is done with the formalism of contraction-extraction coproducts exposed in \cite{Foissy41}.
If $G$ is a hypergraph and $\sim$ is an equivalence on $V(G)$, we denote by $G/\sim$ the hypergraph
which set of vertices is $V(G)/\sim$ and which edges are the nontrivial $\pi_\sim(e)$ for $e$ edge of $G$, 
where $\pi_\sim:V(G)\longrightarrow V(G)/\sim$ is the canonical surjection.
For $\leftthreetimes \in \{\subset,\cap\}$, we denote by $G\mid_\leftthreetimes \sim$ the disjoint union of
induced subgraphs $G_{\mid_\leftthreetimes C}$ with $C\in V(G)/\sim$. We finally shall write that $\sim\in \eq_\leftthreetimes[G]$
if each $G_{\mid_\leftthreetimes C}$ is a connected hypergraph. We then can define
\[\delta^{(\leftthreetimes)}(G)=\sum_{\sim\in \eq_\leftthreetimes[G]} G/\sim\otimes G\mid_\leftthreetimes \sim.\]
For example, if $n\geq 2$,
\begin{align*}
\delta^{(\subset)}(T_n)&=T_n\otimes T_1^n+T_1 \otimes T_n,\\
\delta^{(\cap)}(T_n)&=\sum_{n=1k_1+\ldots+nk_n} \frac{n!}{1!^{k_1}\ldots n!^{k_n} k_1!\ldots k_n!}
T_{k_1+\ldots+k_n}\otimes T_1^{k_1}\ldots T_n^{k_n}.
\end{align*}
We then obtain a double bialgebra $(\calF[\bfH],m,\Delta^{(\leftthreetimes,\leftthreetimes)},\delta^{(\leftthreetimes)})$,
that is to say:
\begin{itemize}
\item $(\calF[\bfH],m,\delta^{(\leftthreetimes)})$ is a bialgebra.
\item $(\calF[\bfH],m,\Delta^{(\leftthreetimes,\leftthreetimes)})$ is a bialgebra in the category of right comodules 
over the bialgebra $(\calF[\bfH],m,\delta^{(\leftthreetimes)})$, with the coaction $\delta^{(\leftthreetimes)}$. 
\end{itemize}
In particular, this implies the compatibility
\[(\Delta^{(\leftthreetimes,\leftthreetimes)}\otimes \id)\circ \delta^{(\leftthreetimes)}
=m_{1,3,24}\circ (\delta^{(\leftthreetimes)}\otimes \delta^{(\leftthreetimes)})\circ \Delta^{(\leftthreetimes,\leftthreetimes)},\]
where
\[m_{1,3,24}:\left\{\begin{array}{rcl}
\calF[\bfH]^{\otimes 4}&\longrightarrow&\calF[\bfH]^{\otimes 3}\\
a_1\otimes a_2\otimes a_3\otimes a_4&\longmapsto&a_1\otimes a_3\otimes a_2a_4.
\end{array}\right.\]
The coproduct $\delta^{(\subset)}$ is different from the coproduct $\delta$ of \cite{Ebrahimi-Fard2022}, the difference coming from
a different gestion of the contractions, seen as equivalences on the set of vertices here, and seen as contractions of edges in \cite{Ebrahimi-Fard2022}.
We did not find a convenient second coproduct for $\Delta^{(\subset,\cap)}$, but we have negative results about it (Proposition \ref{prop2.6} and Corollary \ref{cor2.7}).\\

These results have interesting consequences. Let us fix  $\leftthreetimes \in \{\subset,\cap\}$. A polynomial invariant of $\calF[\bfH]$ is any Hopf algebra morphism from 
$(\calF[\bfH],m,\Delta^{(\leftthreetimes,\leftthreetimes)})$ to $(\K[X],m,\Delta)$, where $\Delta$ is the coproduct defined by
\[\Delta(X)=X\otimes 1+1\otimes X.\]
The following results have been proved in \cite{Foissy36,Foissy37,Foissy40}:
\begin{enumerate}
\item There exists a unique polynomial invariant $P_\leftthreetimes$, that is to say a map from $\calF[\bfH]$ to $\K[X]$,
 which is also compatible with both bialgebraic structures, the second coproduct of $\K[X]$ being defined by
 $\delta(X)=X\otimes X$. 
\item Moreover, any polynomial invariant can be obtained from $P_\leftthreetimes$ by an action $\leftsquigarrow$ 
of the monoid $\Char(\calF[\bfH])$ of characters of the bialgebra $(\calF[\bfH],m,\delta^{(\leftthreetimes)})$. Denoting by
$\calP_\leftthreetimes$ the set of polynomial invariants of $(\calF[\bfH],m,\Delta^{(\leftthreetimes,\leftthreetimes)})$,
the following maps are two bijections, inverse one from the other:
\begin{align*}
&\left\{\begin{array}{rcl}
\Char(\calF[\bfH])&\longrightarrow&\calP_\leftthreetimes\\
\lambda&\longmapsto&P_\leftthreetimes \leftsquigarrow_\leftthreetimes \lambda=(P_\leftthreetimes \otimes \lambda)\circ \delta^{(\leftthreetimes)},
\end{array}\right. \\
&\left\{\begin{array}{rcl}
\calP_\leftthreetimes&\longrightarrow&\Char(\calF[\bfH])\\
\phi&\longmapsto&\left\{\begin{array}{rcl}
\calF[\bfH]&\longrightarrow&\K\\
G&\longmapsto&\phi(G)(1).
\end{array}\right.\end{array}\right.
\end{align*} 
\item The antipode of $(\calF[\bfH],m,\Delta^{(\leftthreetimes,\leftthreetimes)})$, denoted by $S_\leftthreetimes$,
is given by
\[S_\leftthreetimes=({P_\leftthreetimes}_{\mid X=-1}\otimes \id)\circ \delta^{(\leftthreetimes)}.\]
\end{enumerate}
We prove in Proposition \ref{prop2.1} that, if $\leftthreetimes=\subset$, for any hypergraph $G$, $P_\subset(G)$ is a polynomial such that for any $N\geq 0$,
$P_\subset(G)(N)$ is the number of $N$-colourings of $G$, that is to say maps $f:V(G)\longrightarrow\{1,\ldots,N\}$,
such that $f$ is not constant on any non trivial edge of $G$; if $\leftthreetimes=\cap$, for any hypergraph $G$, $P_\cap(G)$ is a polynomial such that for any $N\geq 0$,
$P_\cap(G)(N)$ is the number of $N$-colourings of $G$ 
such that $f$ is injective on any edge of $G$. 
The polynomial $P_\cap(G)$ is in fact the chromatic polynomial of the graph $\Gamma(G)$ obtained by replacing any hyperdge of $G$ 
by a complete graph with the same set of vertices.
The polynomial $P_\subset(G)$ is generally not the chromatic polynomial of a graph.
It seems that its first appearance can be found in \cite{Helgason74}, see also \cite{Borowiecki2000,Dohmen95,Tomescu98,ZhangDong2017}.
It is denoted by $\chi_{E,V}$ in \cite{Ebrahimi-Fard2022}. This method cannot be applied to $(\calF[\bfH],m,\Delta^{(\cap,\subset)})$ by lack of the second coproduct.
Anyway, it is still possible to define a chromatic polynomial invariant $P_{\cap,\subset}$, which plays the role of the $P_\leftthreetimes$
in the sense that for any hypergraph $G$,
\[P_{\cap,\subset}(G)(1)=P_\subset(G)(1)=P_\cap(G)(1)=\begin{cases}
1\mbox{ if $G$ has no non trivial edge},\\
0\mbox{ otherwise}.
\end{cases}\]
We prove in Proposition \ref{prop2.1} that this invariant counts the number of colourings $f$ such that on any edge of $G$, 
the maximum of $f$ is obtained exactly one time:
this is the chromatic polynomial of \cite{Aval2020,Aval2019}.  \\ 

In order to find the antipode, we need to consider values of $P_\leftthreetimes$ at $-1$. 
As for graphs (Stanley's theorem), this is related to acyclic orientations. Here, an acyclic orientation
is a partial quasi-order on the vertices, satisfying certain conditions, see Definition \ref{defi2.8}. We then obtain interpretations of $P_\subset(G)(-1)$ and $P_\cap(G)(-1)$ 
in terms of these orientations (Theorem \ref{theo2.10}), and this is used to give explicit formulas for the antipodes 
of $(\calF[\bfH],m,\Delta^{(\cap,\cap)})$ and $(\calF[\bfH],m,\Delta^{(\subset,\subset)})$, see Corollary \ref{cor2.13}.
A combinatorial interpretation of $P_{\cap,\subset}(G)(-1)$ is also given in Theorem \ref{theo2.10}
and the antipode of $(\calF[\bfH],m,\Delta^{(\cap,\subset)})$ is described, with the help of Takeuchi's formula, in proposition \ref{prop2.15}. 

Using the inverse of a particular character, we give a new proof of a formula on the coefficients of the chromatic polynomial 
$P_\subset(G)$, which can be found in \cite{Tomescu98,ZhangDong2017}, see Proposition \ref{prop2.19}.
We also give some results on decorated versions of $\calF[\bfH]$, where the space of decorations is taken into a commutative and cocommutative bialgebra.
This allows to replace $\K[X]$ by a quasishuffle algebra, and the unique double bialgebra morphism replacing the chromatic polynomials are described in Propositions \ref{prop2.20}
and \ref{prop2.21}. They are also based on colourings of graphs. \\

The last section of this text is devoted to multi-complexes. These objects, introduced in \cite{Iovanov2022}, generalize graphs,
multigraphs, hypergraphs, $\Delta$-complexes, and simplicial complexes. We prove that the bialgebraic structure
of \cite{Iovanov2022} can be extended to a double bialgebra structure, and that $(\calF[\bfH],m,\Delta^{(\subset,\subset)}, \delta^{(\subset)})$ is a quotient of this structure. Consequently, the unique polynomial invariant of multi-complexes
compatible with both coproducts factorizes through the chromatic polynomial $P_\subset$ of the underlying hypergraphs,
which allows to give formulas for the antipode and the eulerian projector for mutli-complexes.\\

\textbf{Acknowledgements}. 
The author acknowledges support from the grant ANR-20-CE40-0007
\emph{Combinatoire Algébrique, Résurgence, Probabilités Libres et Opérades}.\\

\begin{notation} \begin{enumerate}
\item We denote by $\K$ a commutative field of characteristic zero. Any vector space in this field will be taken over $\K$.
\item For any $N\in \N$, we denote by $[N]$ the set $\{1,\ldots,N\}$. In particular, $[0]=\emptyset$.
\item If $(C,\Delta)$ is a (coassociative but not necessarily counitary) coalgebra, 
we denote by $\Delta^{(n)}$ the $n$-th iterated coproduct of $C$:
$\Delta^{(1)}=\Delta$ and if $n\geq 2$,
\[\Delta^{(n)}=\left(\Delta \otimes \id^{\otimes (n-1)}\right)\circ \Delta^{(n-1)}:C\longrightarrow C^{\otimes (n+1)}.\]
\item If $(B,m,\Delta)$ is a bialgebra of unit $1_B$ and of counit $\varepsilon_B$, let us denote by $B_+=\ker(\varepsilon_B)$ 
its augmentation ideal. We define a coproduct on $B_+$ by
\begin{align*}
&\forall x\in B_+,&\tdelta(x)=\Delta(x)-x\otimes 1_B-1_B\otimes x.
\end{align*}
Then $(B_+,\tdelta)$ is a coassociative (not necessarily counitary) coalgebra. 
\item Let $\bfP$ be a species. For any finite set $X$, the vector space associated to $X$ by $\bfP$ is denoted by $\bfP[X]$.
For any bijection $\sigma:X\longrightarrow Y$  between two finite sets, the linear map associated to $\sigma$ by $\bfP$
is denoted by $\bfP[\sigma]:\bfP[X]\longrightarrow \bfP[Y]$. The Cauchy tensor product of species is denoted by $\otimes$:
if $\bfP$ and $\bfQ$ are two species, for any finite set $X$,
\[\bfP\otimes \bfQ[X]=\bigoplus_{X=Y\sqcup Z} \bfP[Y]\otimes \bfQ[Z].\]
If $\sigma:X\longrightarrow Y$  is a bijection between two finite sets, then 
\[\bfP\otimes \bfQ[\sigma]=\bigoplus_{X=Y\sqcup Z} \bfP[\sigma_{\mid Y}]\otimes \bfQ[\sigma_{\mid Z}].\]
A twisted algebra (resp. coalgebra, bialgebra) is an algebra (resp. coalgebra, bialgebra) in the symmetric monoidal category
of species with the Cauchy tensor product. We refer to \cite{Foissy39,Foissy41} for details and notations on algebras,
coalgebras and bialgebras in the category of species.
\item Let $V$ be a vector space. The $V$-coloured Fock functor $\calF_V$, defined in \cite[Definition 3.2]{Foissy41},
 sends any species $\bfP$ to
\begin{align*}
\calF_V[\bfP]&=\bigoplus_{n=0}^\infty \coinv(V^{\otimes n}\otimes \bfP[n])\\
&=\bigoplus_{n=0}^\infty \frac{V^{\otimes n}\otimes \bfP[n]}
{\vect(v_1\ldots v_n \otimes \bfP[\sigma](p)-v_{\sigma(1)}\ldots v_{\sigma(n)}\otimes p\mid
\sigma\in \sym_n,\: p\in \bfP[n],\:v_1,\ldots,v_n\in V)}\\
&=V^{\otimes n}\otimes_{\sym_n} \bfP[n].
\end{align*}
When $V=\K$, we obtain the bosonic Fock functor of \cite{Aguiar2010}:
\[\calF[\bfP]=\bigoplus_{n=0}^\infty \coinv(\bfP[n])
=\bigoplus_{n=0}^\infty \frac{\bfP[n]}{\vect(\bfP[\sigma](p)-p\mid \sigma\in \sym_n,\: p\in \bfP[n])}.\]
\end{enumerate}\end{notation}

\section{Twisted bialgebras of hypergraphs}

\subsection{Definitions}

\begin{defi}\label{defi1.1}
A hypergraph is a family $G=(V(G),E(G))$, where $V(G)$ is a finite set, called the set of vertices of $G$,
and $E(G)$ is a subset of $\mathcal{P}(V(G))$, called the set of edges of $G$.  
For the sake of simplicity, for any hypergraph $G$ we shall consider in this article, we shall assume that:
\begin{itemize}
\item $\emptyset \in E(G)$.
\item For any $x\in V(G)$, $\{x\}\in E(G)$.
\end{itemize}
If $G$ is a hypergraph, we shall denote the set of its nontrivial edges by
\[E^+(G)=\{e\in E(G)\mid |e|\geqslant 2\}.\]
Under our assumption,
\[E(G)=E^+(G)\sqcup\{\emptyset\}\sqcup \{\{x\},\: x\in V(G)\}.\]
If $I$  is a finite set, we shall denote by $\calH[I]$ the set of hypergraphs $G$ such that $V(G)=I$. 
This defines a  set species $\calH$. The linearization of this set species is denoted by $\bfH$: for any finite set $I$,
\[\bfH[I]=\vect(\calH[I]).\]
\end{defi}

\begin{remark} \label{rk1.1} \begin{enumerate}
\item A (simple) graph is a hypergraph $G$ such that for any $e\in E^+(G)$, $|e|=2$. This defines 
a set subspecies of $\calH$ denoted by $\calG_s$ and a subspecies of $\bfH$ denoted by $\bfG_s$. 
This species of graphs and its bialgebraic structures are studied in \cite{Foissy40,Foissy41,Foissy43}. 
\item  If $I$ is a finite set of cardinality $n$, then
\[|\calH[I]|=2^{\displaystyle \sum_{k=2}^n \binom{n}{k}}=2^{2^n-n-1}.\]
This is the de Bruijn's sequence, entry A016031 in the OEIS \cite{Sloane}. 
\[\begin{array}{|c||c|c|c|c|c|c|}
\hline |I|&1&2&3&4&5&6\\
\hline \hline |\calH[I]|&
1&2&16&2048&67108864&144115188075855872\\
\hline\end{array}\]
The following is the number $h_n$ of isoclasses of hypergraphs according to the number of vertices $n$:
\[\begin{array}{|c||c|c|c|c|c|c|}
\hline n&1&2&3&4&5&6\\
\hline \hline h_n&1&2&8&180&612032&200253854316544\\
\hline\end{array}\]
This is sequence A317794 of the OEIS \cite{Sloane}. 
\end{enumerate}\end{remark}

\subsection{Twisted bialgebras of hypergraphs}

Let us now define four twisted bialgebra structures on $\bfH$, with the help of two different notions of induced sub-hypergraphs.

\begin{notation}
\begin{enumerate}
\item Let $G\in \calH[X]$ and $I\subseteq X$. 
\begin{enumerate}
\item $G_{\mid_\subset I}$ is the hypergraph such that 
\begin{align*}
V(G_{\mid_\subset I})&=I,&E(G_{\mid_\subset I})&=\{e\in E(G)\mid e\subset I\}.
\end{align*}
\item $G_{\mid_\cap I}$ is the hypergraph such that 
\begin{align*}
V(G_{\mid_\cap I})&=I,&E(G_{\mid_\cap I})&=\{e\cap I\mid e\in E(G)\}.
\end{align*}
\end{enumerate}
Thanks to the conditions we imposed on hypergraphs, both $G_{\mid_\subset I}$ and $G_{\mid_\cap I}$ belong to $\calH[I]$.
\item Let $X$ and $Y$ be two disjoint sets, $G\in \calH[X]$ and $G'\in \calH[Y]$. 
Then $GG'$ is the hypergraph such that
\begin{align*}
V(GG')&=X\sqcup Y,&E(GG')&=E(G)\sqcup E(G').
\end{align*}
This defines an element of $\calH[X\sqcup Y]$. 
\end{enumerate}
\end{notation}

\begin{lemma} \label{lemmecombi}
Let $\leftthreetimes,\rightthreetimes\in \{\cap,\subset\}$. 
\begin{enumerate}
\item Let $G$ be a hypergraph and $X\subseteq Y\subseteq G$. Then
\[(G_{\mid_\leftthreetimes Y})_{\mid_\leftthreetimes X}=G_{\mid_\leftthreetimes X}.\]
\item Let $G$ be a hypergraph and $I,J,K\subseteq V(G)$ such that $V(G)=I\sqcup J\sqcup K$. Then
\[(G_{\mid_\leftthreetimes J\sqcup K})_{\mid_\rightthreetimes J}=(G_{\mid_\rightthreetimes I\sqcup J})_{\mid_\leftthreetimes J}.\]
\end{enumerate}
\end{lemma}

\begin{proof}
1. We first consider the case $\leftthreetimes=\cap$. Then
\begin{align*}
E((G_{\mid_\cap Y})_{\mid_\cap X})&=\{e\cap X\mid e\in E(G_{\mid_\cap Y})\}\\
&=\{e\cap Y\cap X\mid e\in E(G)\}\\
&=\{e\cap X\mid e\in E(G)\}\\
&=E(G_{\mid_\cap X}).
\end{align*}
So $(G_{\mid_\cap Y})_{\mid_\cap X}=G_{\mid_\cap X}$. Let us the consider the case $\leftthreetimes=\subset$.
\begin{align*}
E((G_{\mid_\subset Y})_{\mid_\subset X})&=\{e \in E(G_{\mid_\subset Y})\mid e\subset X\}\\
&=\{e \in E(G)\mid e\subset X,\: e\subset Y\}\\
&=\{e \in E(G)\mid e\subset X\}\\
&=E(G_{\mid_\subset X}).
\end{align*}
So $(G_{\mid_\subset Y})_{\mid_\subset X}=G_{\mid_\subset X}$.\\

2. If $\leftthreetimes=\rightthreetimes$, both $(G_{\mid_\leftthreetimes J\sqcup K})_{\mid_\rightthreetimes J}$ and $(G_{\mid_\rightthreetimes I\sqcup J})_{\mid_\leftthreetimes J}$
are equal to $G_{\mid_\leftthreetimes J}$ by the first point. Let us now consider the case $\leftthreetimes=\cap$ and $\rightthreetimes=\subset$. 
\begin{align*}
E((G_{\mid_\cap J\sqcup K})_{\mid_\subset J})&=\{e\in E(G_{\mid_\cap J\sqcup K})\mid e\subseteq J\}\\
&=\{e\cap (J\sqcup K)\mid e\in E(G),\: e\cap (J\sqcup K)\subseteq J\}\\
&=\{e\cap J\mid e\in E(G),\: e\cap K=\emptyset\}.\\
E((G_{\mid_\subset I\sqcup J})_{\mid_\cap J})&=\{e\cap J\mid e\in E(G_{\mid_\subset I\sqcup J})\}\\
&=\{e\cap J\mid e\in E(G),\: e\subseteq I\sqcup J\}\\
&=\{e\cap J\mid e\in E(G),\: e\cap K=\emptyset\}.
\end{align*}
Therefore, $(G_{\mid_\cap J\sqcup K})_{\mid_\subset J}=(G_{\mid_\subset I\sqcup J})_{\mid_\cap J}$. 
By symmetry of $I$ and $K$, this also gives the proof for  $\leftthreetimes=\subset$ and $\rightthreetimes=\cap$. 
\end{proof}

\begin{prop}\label{prop1.2}
Let $\leftthreetimes,\rightthreetimes\in \{\cap,\subset\}$. 
We define a twisted bialgebra structure $(\bfH,m,\Delta^{(\leftthreetimes,\rightthreetimes)})$ on $\bfH$  by the following:
\begin{itemize}
\item For any $G\in \calH[X]$, for any $G'\in \calH[Y]$, $m_{X,Y}(G\otimes G')=GG'$.
\item For any $G\in \calH[I\sqcup J]$,
$\Delta^{(\leftthreetimes,\rightthreetimes)}_{I,J}(G)=G_{\mid_\leftthreetimes I}\otimes G_{\mid_\rightthreetimes J}$.
\end{itemize}
The coopposite coproduct of $\Delta^{(\leftthreetimes,\rightthreetimes)}$ is $\Delta^{(\rightthreetimes,\leftthreetimes)}$. 
Consequently, $\Delta^{(\subset,\subset)}$ and $\Delta^{(\cap,\cap)}$ are cocommutative.
\end{prop}

\begin{proof}
The product $m$ is obviously associative and its unit is the empty hypergraph. \\

Let  $I,J,I',J'$ be finite sets such that $I'\sqcup J'=I\sqcup J$, and let $G\in \calH[I']$ and $G'\in \calH[J']$.
As the edges of $GG'$ are included in $I$ or in $J$,
\begin{align*}
\Delta_{I,J}^{(\leftthreetimes,\rightthreetimes)}\circ m_{I',J'}(G\otimes G')
&=(GG')_{\mid_\leftthreetimes I}\otimes (GG')_{\mid T J}\\
&=G_{\mid_\leftthreetimes I\cap I'} G'_{\mid_\leftthreetimes I\cap J'}\otimes G_{\mid_\rightthreetimes J\cap I'} 
G'_{\mid_\rightthreetimes J\cap J'}\\
&=(m_{I\cap I',I'\cap J'}\otimes m_{J\cap I', J\cap J'})\circ
(\id_{\bfH[I\cap I']}\otimes c_{\bfH[J\cap I'],\bfH[I\cap J']}\otimes \id_{\bfH[J\cap J']})\\
&\circ(\Delta^{(\leftthreetimes,\rightthreetimes)}_{I\cap I',J\cap I'}\otimes 
\Delta^{(\leftthreetimes,\rightthreetimes)}_{I\cap J',J\cap J'})(G\otimes G'),
\end{align*}
so $\Delta^{(\leftthreetimes,\rightthreetimes)}$ is an algebra morphism.\\

Let us now prove the coassociativity of $\Delta^{(\leftthreetimes,\rightthreetimes)}$. If  $G\in \calH[I\sqcup J\sqcup K]$, by Lemma \ref{lemmecombi}, first item,
\begin{align*}
(\Delta^{(\leftthreetimes,\rightthreetimes)}_{I,J}\otimes \id)\circ \Delta^{(\leftthreetimes,\rightthreetimes)}_{I\sqcup J,K}(G)
&=(G_{\mid_\leftthreetimes I\sqcup J})_{\mid_\leftthreetimes I}\otimes (G_{\mid_\leftthreetimes I\sqcup J})_{\mid_\rightthreetimes J}
\otimes G_{\mid_\rightthreetimes K}\\
&=G_{\mid_\leftthreetimes  I}\otimes (G_{\mid_\leftthreetimes I\sqcup J})_{\mid_\rightthreetimes J} \otimes G_{\mid_\rightthreetimes K},\\
(\id \otimes \Delta^{(\leftthreetimes,\rightthreetimes)}_{J,K})\circ \Delta^{(\leftthreetimes,\rightthreetimes)}_{I,J\sqcup K}(G)
&=G_{\mid_\leftthreetimes  I}\otimes (G_{\mid_\rightthreetimes J\sqcup K})_{\mid_\leftthreetimes J}\otimes (G_{\mid_\rightthreetimes J\sqcup K})_{\mid_\rightthreetimes K}\\
&=G_{\mid_\leftthreetimes  I}\otimes (G_{\mid_\rightthreetimes J\sqcup K})_{\mid_\leftthreetimes J}\otimes G_{\mid_\rightthreetimes K}.
\end{align*}
By Lemma \ref{lemmecombi}, second item, $(G_{\mid_\leftthreetimes I\sqcup J})_{\mid_\rightthreetimes J}=(G_{\mid_\rightthreetimes J\sqcup K})_{\mid_\leftthreetimes J}$,
so $\Delta^{(\leftthreetimes,\rightthreetimes)}$ is coassociative. These four coproducts share the same counit, defined by $\varepsilon(\emptyset)=1$. 
\end{proof}

\begin{example}\label{ex1.1}
For any finite set $X$ with at least two elements, let us denote by $T_X$ the hypergraph which vertices set is $X$,
with $X$ as a unique nontrivial edge. For any finite nonempty disjoint sets $I$ and $J$,
\begin{align*}
\Delta^{(\subset,\subset)}_{I,J}(T_{I\sqcup J})&=\prod_{x\in I} T_{\{x\}}\otimes \prod_{y\in J} T_{\{y\}},&
\Delta^{(\cap,\cap)}_{I,J}(T_{I\sqcup J})&=T_I\otimes T_J,\\
\Delta^{(\cap,\subset)}_{I,J}(T_{I\sqcup J})&=T_I\otimes \prod_{y\in J} T_{\{y\}},&
\Delta^{(\subset,\cap)}_{I,J}(T_{I\sqcup J})&=\prod_{x\in I} T_{\{x\}}\otimes T_J.
\end{align*}
\end{example}

\begin{remark} 
\begin{enumerate}
\item As seen in Remark \ref{rk1.1}, graphs are hypergraphs, so $\bfH$ contains a subspecies of graphs, 
which is a twisted subbialgebra. Moreover, if $G$ is a graph,
\[\Delta^{(\subset,\subset)}(G)=\Delta^{(\cap,\subset)}(G)
=\Delta^{(\subset,\cap)}(G)=\Delta^{(\cap,\cap)}(G).\]
We recover the twisted bialgebra of graphs of \cite{Foissy40,Foissy41,Foissy43}.
\item Let $I_1,\ldots,I_n$ be disjoint sets. For any hypergraph $G\in \calH[I_1\sqcup \ldots \sqcup I_n]$,
\begin{align*}
\Delta^{(\leftthreetimes,\leftthreetimes)}_{I_1,\ldots,I_n}(G)&=G_{\mid_\leftthreetimes I_1}
\otimes \ldots \otimes G_{\mid_\leftthreetimes I_n},
\end{align*}
whereas
\begin{align*}
\Delta^{(\subset,\cap)}_{I_1,\ldots,I_n}(G)&=G_{\mid^{(1)} I_1}\otimes \ldots \otimes G_{\mid^{(n)} I_n},
\end{align*}
where for any $p\in [n]$, 
\begin{align*}
V(G_{\mid^{(p)} I_p})&=I_p,&E(G_{\mid^{(p)} I_p})=&\{e\cap I_p\mid e\in E(G),\: e\subseteq I_1\sqcup \ldots \sqcup I_p\}.
\end{align*}
In other words, the nonempty edges of $G_{\mid^{(p)} I_p}$ are the sets $e\cap I_p$, where $e$ runs among the edges of $G$
such that 
\[\max\{i\in [n]\mid e\cap I_i\neq \emptyset\}=p.\]
\end{enumerate}\end{remark}

\begin{notation}
In the following, we shall simply write $\Delta^{(\leftthreetimes)}$ for $\Delta^{(\leftthreetimes,\leftthreetimes)}$ for $\leftthreetimes\in \{\subset,\cap\}$.
\end{notation}

\subsection{Contraction-extraction coproducts}

In order to define double bialgebras of graphs, we shall use here the formalism of contraction-extraction coproducts of \cite{Foissy41}.
We introduce for this contractions of hypergraphs, with connectedness constraints.

\begin{defi}
Let $G$ be a hypergraph. A path in $G$ is a sequence $(x_0,\ldots,x_k)$ of vertices of $G$
such that for any $i\in [k]$, there exists an edge $e\in E(G)$ containing both $x_{i-1}$ and $x_i$. 
We shall say that $G$ is connected if for any $x,y\in V(G)$, there exists a path in $G$ from $X$ to $Y$.
Any hypergraph $G$ can be uniquely  written as the product of connected hypergraphs, called the connected components of $G$.
\end{defi}

\begin{notation}  
We use the notations of \cite[Notations 2.1]{Foissy41} for the equivalence relations. 
For any finite set $X$, $\eq[X]$ is the set of equivalence  relations on $X$. It is partially ordered by the refinement order: 
if $\sim,\sim'\in \eq[X]$, then 
\[\sim\leq \sim'\Longleftrightarrow (\forall x,y\in X,\: x\sim' y\Longrightarrow x\sim y).\]
If $\sim'\in \eq[X]$, then $\{\sim\in \eq[X]\mid \sim \leq \sim'\}$ and  $\eq[X/\sim']$ are in bijection, via the map
sending $\sim$ to $\overline{\sim}$ defined by
\[\overline{x}\overline{\sim}\overline{y}\Longleftrightarrow x\sim y.\]
We identify in this way  $\{\sim\in \eq[X]\mid \sim \leq \sim'\}$ and  $\eq[X/\sim']$.
\end{notation}

\begin{defi}
Let $G\in \calH[X]$, $\sim\in \eq[X]$ and let $\leftthreetimes \in \{\subset,\cap\}$.
\begin{itemize}
\item We define the hypergraph $G/\sim\in \calH[X/\sim]$ by
\begin{align*}
V(G/\sim)&=X/\sim,&
E(G/\sim)&=\{\pi_\sim(e)\mid e\in E(G)\},
\end{align*}
where $\pi_\sim:X\longrightarrow X/\sim$ is the canonical surjection. By the conditions we imposed in Definition \ref{defi1.1} on hypergraphs,
this is indeed a hypergraph. 
\item We define the hypergraph $G\mid_\leftthreetimes \sim\in \calH[X]$ by
\[G\mid_\leftthreetimes \sim=\prod_{C\in X/\sim} G_{\mid_\leftthreetimes C}.\]
\item We shall say that $\sim\in \eq_\leftthreetimes[G]$ if for any class $C$ of $\sim$, $G_{\mid_\leftthreetimes C}$ is connected. 
\end{itemize}\end{defi}

\begin{remark} For any hypergraph $G$ and $\sim\in \eq[V(G)]$, $V(G\mid_\cap \sim)=V(G\mid_\subset \sim)=V(G)$, and
\begin{align*}
E(G\mid_\cap \sim)&=\{e\cap C\mid e\in E(G),\: C\in V(G)/\sim\},\\
E(G\mid_\subset \sim)&=\{e\in E(G)\mid |\pi_\sim(e)|\leq 1\}.
\end{align*}
\end{remark}

By definition, if $\sim \in \eq_\leftthreetimes[G]$, the connected components of $G_{\mid_\leftthreetimes C}$ 
are the classes of $\sim$. 

\begin{theo} \label{theo1.4}
Let $\leftthreetimes \in \{\cap,\subset\}$. 
For any hypergraph $G\in \calH[X]$ and for any $\sim\in \eq[X]$, we put
\[\delta_\sim^{(\leftthreetimes)}(G)=\begin{cases}
G/\sim\otimes \:G\mid_\leftthreetimes \sim\mbox{ if }\sim\in \eq_\leftthreetimes[G],\\
0\mbox{ otherwise}.
\end{cases}\]
This defines a contraction-extraction coproduct on $\bfH$ in  the sense of \cite{Foissy41}, 
compatible with $m$ and $\Delta^{(\leftthreetimes)}$. 
\end{theo}

Let us start the proof of this theorem  with a combinatorial lemma.

\begin{lemma}\label{lem1.5}
Let $G\in \bf[X]$ and $\leftthreetimes \in \{\subset,\cap\}$.
\begin{enumerate}
\item If  $\sim\leqslant \sim'\in \eq[X]$, then the hypergraphs $(G/\sim')/\overline{\sim}$ and $G/\sim$ are equal.
\item Let $\sim \in \eq[X]$. Then 
\begin{align*}
\sim \in \eq_\leftthreetimes[G]&\Longleftrightarrow\mbox{the connected components of 
$G\mid_\leftthreetimes \sim$ are the classes of $\sim$}.
\end{align*}
\item Let $\sim\in \eq_\leftthreetimes[G]$. The connected components of $G/\sim$ are the images by $\pi_\sim$ 
of the connected components of $G$. 
\item Let $\sim\leqslant \sim'\in \eq[X]$.
Then 
\begin{align*}
\sim'\in \eq_\leftthreetimes[G]\mbox{ and }\overline{\sim}\in \eq_\leftthreetimes[G/\sim']
&\Longleftrightarrow \sim \in \eq_\leftthreetimes[G] \mbox{ and }\sim'\in \eq_\leftthreetimes[G\mid_\leftthreetimes \sim].
\end{align*}
If this holds, $(G/\sim')\mid_\leftthreetimes \overline{\sim} =(G\mid_\leftthreetimes\sim)/\sim'$. 
\end{enumerate}
\end{lemma}

\begin{proof} 1. Firstly, $V((G/\sim')/\overline{\sim})=V(G)/\sim=V(G/\sim)$ and secondly, 
\begin{align*}
E((G/\sim')/\overline{\sim})&=\{\pi_{\overline{\sim}}\circ \pi_{\sim'}(e)\mid e\in E(G)\}
=\{\pi_\sim(e)\mid e\in E(G)\}=E(G/\sim).
\end{align*}
Hence, $(G/\sim')/\overline{\sim}=G/\sim$. \\

2. Immediate consequence of  the definition of $\eq_\leftthreetimes[G]$, as $\pi_\sim=\pi_{\overline{\sim}}\circ \pi_{\sim'}$. \\

3. By definition of the connectivity, if $H$ is a connected hypergraph and $\sim\in \eq[V(H)]$, then $H/\sim$ is connected.
Consequently, if $H$ is a connected component of $G$, $\pi_\sim(H)$ is connected, so is included in a connected component
of $G/\sim$: we proved that the connected components of $G/\sim$ are union of images by $\pi_\sim$ of connected components
of $G$. 

Let us consider the equivalence $\sim_G$ defined on $V(G)$ by
\[x\sim_G y\mbox{ if there exists a path in $G$ from $x$ to $y$}.\]
By definition, the classes of $G$ are the connected components of $G$, and $\sim_G\in \eq_\leftthreetimes[G]$. 
As $\sim \in \eq_\leftthreetimes[G]$, its classes are connected, so are included in a single connected component of $G$:
$\sim_G\leq \sim$. Therefore, if $x$ and $y$ are in two different connected components of $G$,
then $\pi_\sim(x)\neq \pi_\sim(y)$ and there is no edge containing these two elements in $G/\sim$:
any connected component of $G/\sim$ is included in a single $\pi_\sim(H)$, where $H$ is a connected component of $G$. \\

4. $\Longrightarrow$. Let $C$ be a class of $\sim'$. As $\sim\leqslant \sim'$,
$(G\mid_\leftthreetimes \sim)_{\mid_\leftthreetimes C}=G_{\mid_\leftthreetimes C}$, so is connected  
as $\sim' \in \eq_\leftthreetimes[G]$. Therefore, $\sim'\in \eq_\leftthreetimes[G\mid_\leftthreetimes \sim]$. 

Let $C$ be a class of $\sim$ and let $x,y\in C$. As $\overline{\sim}\in \eq_\leftthreetimes[G/\sim']$,
there exists a path from $\pi_{\sim'}(x)$ to $\pi_{\sim'}(y)$ in $(G/\sim')\mid_\leftthreetimes \overline{\sim}$.
We denote this path by $(\pi_{\sim'}(x_0),\ldots,\pi_{\sim'}(x_k))$. Note that all the elements $\pi_{\sim'}(x_p)$
are $\overline{\sim}$-equivalent, so all the elements $x_p$ are $\sim$-equivalent.
By definition of $G/\sim'$, we can assume that for any $p$, there exists $y_p$ such that $x_p\sim' y_p$, 
and with an edge in $G$ containing both $y_p$ and $x_{p+1}$. As $\sim'\in \eq_\leftthreetimes[G]$,
for any $p$ there exists a path from $x_p$ to $y_p$, with all the vertices being $\sim'$-equivalent, so also $\sim$-equivalent
as $\sim\leqslant \sim'$. Hence, there exists in $G$ a path from $x$ to $y$ with all vertices being $\sim$-equivalent:
$C$ is connected, which proves that $\sim \in \eq_\leftthreetimes[G]$. \\

$\Longleftarrow$. Let $G\in \bf[X]$ and $\sim\leqslant \sim'\in \eq[X]$.
Let us prove that  if  $\sim \in \eq_\leftthreetimes[G]$ and $\sim'\in \eq_\leftthreetimes[G\mid_\leftthreetimes \sim]$,
then $\sim'\in \eq_\leftthreetimes[G]$ and $\overline{\sim}\in \eq_\leftthreetimes[G/\sim']$.

Let $C$ be a class of $\sim'$. Then it is connected in $G\mid_\leftthreetimes \sim$, so also in $G$: 
we proved that $\sim'\in \eq_\leftthreetimes[G]$.
Let $\pi_{\sim'}(C)$ be a class of $\overline{\sim}$: as $\sim\leqslant \sim'$, we can assume that $C$ is a class of $\sim$. 
As $\sim \in \eq_\leftthreetimes[G]$, $C$ is connected. By the third item of Lemma \ref{lem1.5},
 $\pi_{\sim'}(C)$ is connected in $G/\sim'$, so $\overline{\sim}\in \eq_\leftthreetimes[G/\sim']$. 

Let us now prove the equality $(G/\sim')\mid_\subset \overline{\sim}=(G\mid_\subset \sim)/\sim'$. As $\sim\leqslant \sim'$,
\begin{align*}
E((G/\sim')\mid_\subset \overline{\sim})&=\{\pi_{\sim'}(e)\mid e\in E(G),\mbox{ all the elements of $\pi_{\sim'}(e)$ 
are $\overline{\sim}$-equivalent}\}\\
&=\{\pi_{\sim'}(e)\mid e\in E(G),\mbox{ all the elements of $e$ 
are $\sim$-equivalent}\}\\
&=E((G\mid_\subset \sim)/\sim').
\end{align*}
Let us finally prove the equality $(G/\sim')\mid_\cap \overline{\sim}=(G\mid_\cap \sim)/\sim'$.
\begin{align*}
E((G/\sim')\mid_\cap \overline{\sim})&=\{\pi_{\sim'}(e)\cap C\mid e\in E(G),\: C\in V(G/\sim')/\overline{\sim}\}\\
&=\{\pi_{\sim'}(e)\cap \pi_{\sim'}(C)\mid e\in E(G),\: C\in V(G)/\sim\},\\
E((G\mid_\cap \sim)/\sim')&=\{\pi_{\sim'}(e\cap C)\mid e\in E(G),\: C\in V(G)/\sim\}.
\end{align*}
Let $e\in E(G)$ and $C\in V(G)/\sim$. Obviously, $\pi_{\sim'}(e\cap C)\subseteq \pi_{\sim'}(e)\cap \pi_{\sim'}(C)$.
Let $\overline{y}\in \pi_{\sim'}(e)\cap \pi_{\sim'}(C)$. There exists $y'\in e$ and $y''\in C$
such that $y\sim' y'\sim' y''$. As $\sim \leqslant \sim'$, $y\sim y'\sim y''$, so $y'\in C$ and $\overline{y}=\overline{y'}
\in  \pi_{\sim'}(e\cap C)$. We proved that $\pi_{\sim'}(e\cap C)=\pi_{\sim'}(e)\cap \pi_{\sim'}(C)$,
which implies that $(G/\sim')\mid_\cap \overline{\sim}=(G\mid_\cap \sim)/\sim'$.
\end{proof}

\begin{proof} (Theorem \ref{theo1.4}). 
Let us first prove the coassociativity of $\delta^{(\leftthreetimes)}$, see \cite[Definition 2.2, third item]{Foissy41}.
Let $G\in \bf[X]$ and $\sim\leqslant \sim'\in \eq[X]$. Then, by Lemma \ref{lem1.5}, first item,
\begin{align*}
(\delta_\sim^{(\leftthreetimes)}\otimes \id)\circ \delta_{\sim'}^{(\leftthreetimes)}(G)
&=\begin{cases}
(G/\sim')/\overline{\sim}\otimes (G/\sim')\mid_\leftthreetimes\overline{\sim}\otimes G\mid_\leftthreetimes \sim'
\mbox{ if $\sim'\in \eq_\leftthreetimes[G]$ and $\overline{\sim}\in \eq_\leftthreetimes[G/\sim']$}\\
0\mbox{ otherwise}
\end{cases}\\
&=\begin{cases}
G/\sim\otimes (G/\sim')\mid_\leftthreetimes\overline{\sim}\otimes G\mid_\leftthreetimes \sim'
\mbox{ if $\sim'\in \eq_\leftthreetimes[G]$ and $\overline{\sim}\in \eq_\leftthreetimes[G/\sim']$}\\
0\mbox{ otherwise}
\end{cases}\\
(\id \otimes \delta_{\sim'}^{(\leftthreetimes)})\circ \delta_\sim^{(\leftthreetimes)}(G)
&=\begin{cases}
G/\sim \otimes (G\mid_\leftthreetimes \sim)/\sim' \otimes (G\mid_\leftthreetimes \sim)\mid_\leftthreetimes \sim'
\mbox{ if $\sim \in \eq_\leftthreetimes[G]$ and $\sim'\in \eq_\leftthreetimes[G\mid_\leftthreetimes \sim]$},\\
 0\mbox{ otherwise}
\end{cases}\\
&=\begin{cases}
G/\sim \otimes (G\mid_\leftthreetimes \sim)/\sim' \otimes G\mid_\leftthreetimes \sim'
\mbox{ if $\sim \in \eq_\leftthreetimes[G]$ and $\sim'\in \eq_\leftthreetimes[G\mid_\leftthreetimes \sim]$},\\
 0\mbox{ otherwise}.
\end{cases}
\end{align*}
By Lemma \ref{lem1.5},  fourth item,
\[(\delta_\sim^{(\leftthreetimes)}\otimes \id)\circ \delta_{\sim'}^{(\leftthreetimes)}(G)
=(\id \otimes \delta_{\sim'}^{(\leftthreetimes)})\circ \delta_\sim^{(\leftthreetimes)}(G).\]

Let us now prove the multiplicativity of $\delta^{(\leftthreetimes)}$, see \cite[Proposition 2.4]{Foissy41}.
Let $G\in \bfH[X]$ and $G'\in \bfH[Y]$, and $\sim \in \eq[X\sqcup Y]$.
If $\sim\neq \sim_X\sqcup \sim_Y$, because of the connectivity condition, $\sim\notin \eq_\leftthreetimes[GG']$,
so $\delta^{(\leftthreetimes)}_{\sim}(GG')=0$. Otherwise, $\sim\in \eq_\leftthreetimes[GG']$ 
if, and only if,  $\sim_X\in \eq_\leftthreetimes[G]$
and $\sim_Y\in \eq_\leftthreetimes[G']$, and, if this holds:
\begin{align*}
(GG')/\sim&=(G/\sim_X)(G'/\sim_Y),&
(GG')\mid_\leftthreetimes \sim&=(G\mid_\leftthreetimes \sim_X)(G'\mid_\leftthreetimes \sim_Y).
\end{align*}
This implies that $\delta^{(\leftthreetimes)}_\sim(GG')=\delta^{(\leftthreetimes)}_{\sim_X}(G)
\delta^{(\leftthreetimes)}_{\sim_Y}(G')$. \\

Let us prove the compatibility  of $\delta^{(\leftthreetimes)}$ with $\Delta^{(\leftthreetimes)}$, see \cite[Proposition 2.5]{Foissy41}. 
Let $G\in \calH[X\sqcup Y]$, $\sim_X\in \eq[X]$ and $\sim_Y\in \eq[Y]$. We put $\sim=\sim_X\sqcup \sim_Y\in \eq[X\sqcup Y]$. 
\begin{align*}
&(\Delta^{(\leftthreetimes)}_{X/\sim_X,Y/\sim_Y}\otimes \id)\circ \delta_\sim^{(\leftthreetimes)}(G)\\
&=\begin{cases}
(G/\sim)_{\mid_\leftthreetimes X/\sim_X}\otimes (G/\sim)_{\mid_\leftthreetimes Y/\sim_Y}
\otimes G\mid_\leftthreetimes\sim \mbox{ if }\sim\in \eq_\leftthreetimes[G],\\
0\mbox{ otherwise}
\end{cases}\\
&=\begin{cases}
(G_{\mid_\leftthreetimes X})/\sim_X \otimes (G_{\mid_\leftthreetimes Y})/\sim_Y\otimes 
(G\mid_\leftthreetimes \sim)_{\mid_\leftthreetimes X}
(G\mid_\leftthreetimes \sim)_{\mid_\leftthreetimes Y}\mbox{ if $\sim_X\in \eq_\leftthreetimes[G_{\mid_\leftthreetimes X}]$ 
and  $\sim_Y\in \eq_\leftthreetimes[G_{\mid_\leftthreetimes Y}]$}\\
0\mbox{ otherwise}
\end{cases}\\
&=m_{1,3,24}\circ (\delta_{\sim_X}^{(\leftthreetimes)}\otimes \delta_{\sim_Y}^{(\leftthreetimes)})
\circ \Delta_{X,Y}^{(\leftthreetimes)}(G).
\end{align*}

Let us finally prove that $\delta^{(\leftthreetimes)}$ has a counit, see \cite[Definition 2.2, fourth item]{Foissy41}.
For any hypergraph $G\in \calH[X]$, we put
\[\epsilon_\delta[X](G)=\begin{cases}
1\mbox{ if }E^+(G)=\emptyset,\\
0\mbox{ otherwise}.
\end{cases}\]
If $G\in \calH[X]$, let us denote by $\sim_0$ the equivalence on $X$ which classes are the connected components of $G$.
By definition, $\sim_0\in \eq_\leftthreetimes[G]$, $G\mid \sim_0=G$ and $G/\sim_0$ is a hypergraph with no nontrivial edge.
Moreover, if $\sim\in \eq_\leftthreetimes[G]$ is different from $\sim_0$, by the second item of Lemma \ref{lem1.5},
 at least one of the connected component of $G/\sim$ is not reduced to a vertex, so has a nontrivial edge: 
 $\epsilon_\delta[X/\sim](G/\sim)=0$. Hence,
\[(\epsilon_\delta\otimes \id)\circ \delta^{(\leftthreetimes)}(G)=G\mid \sim_0+0=G.\]
Let $\sim_1$ be the equality of $X$. Then $\sim_1\in \eq_\leftthreetimes[G]$, $G\mid \sim_1$ is a hypergraph 
with no nontrivial edge and $G/\sim_1=G$. Moreover, if $\sim\in \eq_\leftthreetimes[G]$ is different from $\sim_1$, 
at least one of its class is not reduced to a vertex, so, as it is connected, has a non trivial edge: $\epsilon_\delta[X](G\mid \sim)=0$.
 Hence,
\[(\id \otimes \epsilon_\delta)\circ \delta^{(\leftthreetimes)}(G)=G/\sim_1+0=G.\]
So $\epsilon_\delta$ is the counit of $\delta$. \end{proof}

\begin{example}
With the notations of Example \ref{ex1.1},
\begin{align*}
\delta^{(\subset)}(T_{\{x\}})&=T_{\{x\}}\otimes T_{\{x\}},\\
\delta^{(\cap)}(T_{\{x\}})&=T_{\{x\}}\otimes T_{\{x\}},\\
\delta^{(\subset)}(T_{\{x,y\}})&=T_{\{x,y\}}\otimes T_{\{x\}}T_{\{y\}}+T_{\{\{x,y\}\}}\otimes T_{\{x,y\}},\\
\delta^{(\cap)}(T_{\{x,y\}})&=T_{\{x,y\}}\otimes T_{\{x\}}T_{\{y\}}+T_{\{\{x,y\}\}}\otimes T_{\{x,y\}},\\
\delta^{(\subset)}(T_{\{x,y,z\}})&=T_{\{x,y,z\}}\otimes T_{\{x\}}T_{\{y\}}T_{\{z\}}+
T_{\{\{x,y,z\}\}}\otimes T_{\{x,y,z\}},\\
\delta^{(\cap)}(T_{\{x,y,z\}})&=T_{\{x,y,z\}}\otimes T_{\{x\}}T_{\{y\}}T_{\{z\}}+
T_{\{\{x,y,z\}\}}\otimes T_{\{x,y,z\}}\\
&+T_{\{\{x,y\},\{z\}\}}\otimes T_{\{x,y\}}T_{\{z\}}+
T_{\{\{x,y\},\{z\}\}}\otimes T_{\{x,z\}}T_{\{y\}}+
T_{\{\{x,y\},\{z\}\}}\otimes T_{\{y,z\}}T_{\{x\}}.
\end{align*}
\end{example}

Consequently, if $V$ is a vector space, we obtain four bialgebra structures on $\calF_V[\bfH]$. As a vector space,
they are generated by isomorphism classes of linearly $V$-decorated hypergraphs, that is to say pairs $(H,d_H)$,
where $H$ is a hypergraph and $d_H:V(G)\longrightarrow V$ is a map, with relations such that these decorations
are linear in any vertex. The product is given by disjoint union. The coproducts are given on any $V$-decorated hypergraph $G$ by
\[\Delta^{(\leftthreetimes,\rightthreetimes)}(G)=\sum_{I\subseteq V(G)} G_{\mid_\leftthreetimes I}
\otimes G_{\mid_\rightthreetimes V(G)\setminus I},\]
where $(\leftthreetimes,\rightthreetimes) \in \{\subset,\cap\}^2$.
Moreover, if $(V,\cdot,\delta_V)$ is a  not necessarily unitary, commutative and cocommutative bialgebra,
we obtain two double bialgebras $(\calF_V[\calH],m,\Delta^{(\leftthreetimes)},\delta^{(\leftthreetimes)})$,
with $\leftthreetimes\in \{\subset,\cap\}$. The coproduct $\delta^{(\leftthreetimes)}$ is defined on any $V$-decorated
hypergraph $G$ by
\[\delta^{(\leftthreetimes)}(G)=\sum_{\sim \in \eq_\leftthreetimes[G]} G/\sim \otimes G\mid_\leftthreetimes \sim,\]
where the vertices of $G/\sim \otimes G\mid_\leftthreetimes \sim$ are decorated in the following way:
denoting by $d_G(x)$ the decoration of the vertex $x\in V(G)$, any vertex $\cl_\sim(x)$ of $G/\sim$
is decorated by the products of elements $d_G(y)'$, where $y\in \cl_\sim(x)$, whereas the vertex $x\in V(G\mid \sim)=V(G)$
is decorated by $d_G(x)''$, and everything being extended by multilinearity of each decoration. 
The counit $\epsilon_\delta$ is given on any mixed graph $G$ by
\[\epsilon_\delta(G)=\begin{cases}
\displaystyle \prod_{x\in V(G)} \epsilon_V\circ d_G(x) \mbox{ if }E^+(G)=\emptyset,\\
0\mbox{ otherwise}.
\end{cases}\]
This construction is functorial in $V$.

In the particular case where $V=\K$, we obtain the bosonic Fock functor $\calF[\bfH]$. As a vector space, a basis is given 
by isomorphisms classes of hypergraphs. It is given four bialgebra structures 
$(\calF[\bfH],m,\Delta^{(\leftthreetimes,\rightthreetimes)})$
and two double bialgebra structures $(\calF[\calH],m,\Delta^{(\leftthreetimes)},\delta^{(\leftthreetimes)})$,
with $\leftthreetimes,\rightthreetimes \in \{\subset,\cap\}$.

\begin{example}\label{ex1.3}
For example, if $T_n$ is the hypergraph with $n$ vertices and a unique nontrivial edge $e$ containing all vertices,
we obtain, for $n\geq 2$, 
\begin{align*}
\Delta^{(\subset,\subset)}(T_n)&=T_n \otimes 1+1\otimes T_n+\sum_{k=1}^{n-1} \binom{n}{k} T_1^k\otimes T_1^{n-k},\\
\Delta^{(\cap,\cap)}(T_n)&=T_n \otimes 1+1\otimes T_n+\sum_{k=1}^{n-1} \binom{n}{k} T_k\otimes T_{n-k},\\
\Delta^{(\cap,\subset)}(T_n)&=T_n \otimes 1+1\otimes T_n+\sum_{k=1}^{n-1} \binom{n}{k} T_k\otimes T_1^{n-k},\\
\Delta^{(\subset,\cap)}(T_n)&=T_n \otimes 1+1\otimes T_n+\sum_{k=1}^{n-1} \binom{n}{k} T_1^k\otimes T_{n-k};\\ 
\delta^{(\subset)}(T_n)&=T_n\otimes T_1^n+T_1 \otimes T_n,\\
\delta^{(\cap)}(T_n)&=\sum_{n=1k_1+\ldots+nk_n} \frac{n!}{1!^{k_1}\ldots n!^{k_n} k_1!\ldots k_n!}
T_{k_1+\ldots+k_n}\otimes T_1^{k_1}\ldots T_n^{k_n}.
\end{align*}\end{example}

\begin{remark}
In \cite{Ebrahimi-Fard2022}, twelve coproducts on hypergraphs are introduced. 
The hypergraphs considered there are more general than ours, as the conditions we impose on 
edges of cardinality $\leq 1$ is not required. Let us denote by $\bfH'$ the set of hypergraphs of \cite{Ebrahimi-Fard2022} and by $\calH'$ the space generated by the isoclasses.
We define a map $\theta$ from $\bfH'$ to $\calF[\bfH]$ by sending any $G\in \calH'$ to:
\begin{itemize}
\item 0 if $G$ has an empty edge or an edge of cardinality 1.
\item The unique hypergraph $\theta(G)\in \calH$ such that $E^+(\theta(G))=E^+(G)$
(that is to say we add the empty set and all the singletons as edges). 
\end{itemize}
It is then not difficult to prove that $\theta$ is a bialgebra morphism from $(\calH',m,\Delta)$ to $(\calF[\bfH],m,\Delta^{(\subset,\subset)})$, 
and from $(\calH',m,\Delta')$ to $(\calF[\bfH],m,\Delta^{(\subset,\cap)})$.
The other coproducts $\Delta^d$, $\Delta^c$ dans $\Delta^{cd}$ of \cite{Ebrahimi-Fard2022}, using duality and complementation, 
do not fit well with our context, because of the restrictions we impose on hypergraphs. 
The coproduct $\delta$ of \cite{Ebrahimi-Fard2022} is not $\delta^{(\subset)}$, as shown by \cite[Example 4.3]{Ebrahimi-Fard2022}.
\end{remark}

\begin{remark}
We assume that for any hypergraph $G$, $\emptyset$ and the singletons $\{v\}$, with $v\in V(G)$, belong to $E(G)$.
We can relax this hypothesis by only assuming that $\emptyset \in E(G)$. The objects we obtain in this way will be called general hypergraphs.
General hypergraphs are identified with hypergraphs decorated by the set $\{0,1\}$: for any general hypergraph $G$,
decorate its vertex $v\in V(G)$ by $1$ if $\{v\}\in E(G)$ and by $0$ otherwise. Therefore, choosing any two-dimensional commutative and 
cocommutative  bialgebra
with a basis $(e_0,e_1)$ gives rise to two double bialgebra structures on generalized hypergraph. 
For example, choosing the product and coproducts defined by
\begin{align*}
e_0\cdot e_0&=e_0,&e_0\cdot e_1&=e_1,&\delta_V(e_0)&=e_0\otimes e_0,\\
e_1\cdot e_0&=e_1&e_1\cdot e_1&=e_1,&\delta_V(e_1)&=e_1\otimes e_1,
\end{align*} 
we obtain coproducts $\Delta^{(\subset)}$ and $\Delta^{(\cap)}$ given by induction of sub-hypergraphs,
cointeracting with coproducts $\delta^{(\subset)}$ and $\Delta^{(\cap)}$ of contractions and extractions.
For the contraction part, the vertex obtained by the identification of a subset $X$ of $V(G)$ is part of an edge of cardinality 1
if, and only if, at least one of the element of $X$ is part of an edge of $G$ of cardinality 1. 
\end{remark}

\begin{prop}
Let $V$ be a (non necessarily unitary) commutative and cocommutative bialgebra. 
For any linearly $V$-decorated hypergraph $(G,d_G)$. The following map is a double bialgebra morphism:
\[\Theta_V:\left\{\begin{array}{rcl}
\calF_V[\bfH]&\longrightarrow&\calF[\bfH]\\
(G,d_G)&\longmapsto&\displaystyle \left(\prod_{x\in V(G)} \epsilon_V\circ d_G(x)\right) G.
\end{array}\right.\]
\end{prop}

\begin{proof}
The counit $\epsilon_V:V\longrightarrow \K$ is a bialgebra morphism. By functoriality, $\Theta_V$ is a double bialgebra morphism. 
\end{proof}

\section{Polynomial invariants}

\subsection{Chromatic polynomials}

From \cite[Theorem 3.9]{Foissy40}, if  $\leftthreetimes \in \{\cap,\subset\}$, there exists a unique morphism 
$P_\leftthreetimes$ of double bialgebras  from  $(\calF_V[\bfH],m,\Delta^{(\leftthreetimes)},\delta^{(\leftthreetimes)})$ 
to the double bialgebra $(\K[X],m,\Delta,\delta)$, with
\begin{align*}
\Delta(X)&=X\otimes 1+1\otimes X,&\delta(X)&=X\otimes X. 
\end{align*}
Let us determine $P_\leftthreetimes$. Let $G\in \calH[X]$ be a nonempty hypergraph. Then, still from \cite[Theorem 3.9]{Foissy40},
\[P_\leftthreetimes(G)=\sum_{k=0}^\infty \left(\epsilon_\delta^{\otimes (k-1)}\circ \left(\tdelta^{(\leftthreetimes)}\right)^{(k-1)}(G)\right) H_k(X),\]
where $H_k$ is the $k$-th Hilbert polynomial:
\[H_k(X)=\frac{X(X-1)\ldots (X-k+1)}{k!}.\]

\begin{prop}\label{prop2.1}
Let $\leftthreetimes \in \{\cap,\subset\}$. 
The unique double bialgebra morphism $P_\leftthreetimes$ from $(\calF[\bfH],m,$ $\Delta^{(\leftthreetimes)},
\delta^{(\leftthreetimes)})$ to $ (\K[X],m,\Delta,\delta)$ sends any hypergraph $G$ to a polynomial $P_\leftthreetimes(G)$ 
such that, for any $N\in \N_{>0}$:
\begin{itemize}
\item $P_\cap(G)(N)$ is the number of maps $f:V(G)\longrightarrow [N]$ such that if $x$ and $y$ are two
distinct elements of an edge $e\in E(G)$, then $f(x)\neq f(y)$.
\item $P_\subset(G)(N)$  is the number of maps $f:V(G)\longrightarrow [N]$ such that for any nontrivial edge $e\in E(G)$,
$f$ takes at least two different values on $e$. 
\end{itemize}
\end{prop}

\begin{proof}
We obtain that
\begin{align*}
P\leftthreetimes(G)&=\sum_{k=1}^\infty \sum_{\substack{V(G)=I_1\sqcup \ldots \sqcup I_k\\I_1,\ldots,I_k\neq \emptyset}}
 \epsilon'_\delta(G_{\mid_\leftthreetimes I_1})\ldots \epsilon'_\delta
(G_{\mid_\leftthreetimes I_k})H_k(X)\\
&=\sum_{k=1}^\infty \sum_{\substack{f:V(G)\longrightarrow [k]\\
\mbox{\scriptsize $f$ surjective}\\\forall i\in [k], \: E^+(G_{\mid_\leftthreetimes f^{-1}(i)})=\emptyset}} H_k(X).
\end{align*}
Hence, for any $N\in \N_{>0}$, $P_\leftthreetimes(G)(N)$ is the number of maps $f:V(G)\longrightarrow [N]$
such that for any $i\in [N]$, $E^+(G_{\mid_\leftthreetimes f^{-1}(i)})=\emptyset$. 

If $S=\cap$, this is equivalent to the fact that $f^{-1}(i)$ contains at most one vertex of any $e\in E^+(G)$,
which gives the interpretation of the proposition. If $S=\subset$, this is equivalent to the fact that any $f^{-1}(i)$
does not contain any $e\in E^+(G)$, which gives the interpretation of the proposition.
\end{proof}

\begin{remark}
If $G$ is a graph, both $P_\cap(G)$ and $P_\subset(G)$ are equal to the chromatic polynomial of $G$. 
\end{remark}

Even without a coproduct $\delta$ making $(\calF_V[\bfH],m,\Delta^{(\subset,\cap)})$ a double bialgebra
(see Proposition \ref{prop2.6}),  we can define a Hopf algebra morphism, recovering the chromatic polynomial of \cite{Aval2020,Aval2019}:

\begin{prop}
For any hypergraph $G$, there exists a polynomial $P_{\subset,\cap}(G)$ such that for any $N\in \N_{>0}$,
$P_{\subset,\cap}(G)(N)$ is the number of maps $f:V(G)\longrightarrow [N]$
such that for any $e \in E^+(G)$, $\max\{f(x)\mid x\in e\}$ is obtained in exactly one element of $e$. 
Then $P_{\subset,\cap}:(\calF[\bfH],m,\Delta^{(\subset,\cap)})\longrightarrow (\K[X],m,\Delta)$ is a Hopf algebra morphism. 

\end{prop}

\begin{proof}
The map $\epsilon_\delta$ is a character of $\calF_V[\bfH]$. Hence, we obtain a bialgebra map
\[P_{\subset,\cap}(G)=\sum_{k=1}^\infty \sum_{\substack{V(G)=I_1\sqcup \ldots \sqcup I_k\\I_1,\ldots,I_k\neq \emptyset}}
\epsilon_\delta(G_{\mid^{(1)} I_1})\ldots \epsilon_\delta(G_{\mid^{(k)} I_k}) H_k(X). \]
In other words,
\begin{align*}
P_{\subset,\cap}(G)&=\sum_{k=1}^\infty\sum_{\substack{f:V(G)\longrightarrow [k]\\
\mbox{\scriptsize $f$ surjective}\\\forall i\in [k], \: E^+(G_{\mid^{(i)} f^{-1}(i)})=\emptyset}} H_k(X).
\end{align*}
By construction, for any $N\in \N_{>0}$, $P_{\subset,\cap}(G)(N)$ is the number of maps $f:V(G)\longrightarrow [N]$
such that for any $i\in [N]$, $E^+(G_{\mid^{(i)} f^{-1}(i)})=\emptyset$. 
this is equivalent to the fact that for any edge $e$, $f^{-1}(\max(f_{\mid e}))\cap e$ does not contain 
any nontrivial  edge of $G$, which means that it is reduced to a single vertex. This gives the interpretation of the proposition.
\end{proof}

\begin{example}\label{ex2.1}
Let us use the notations of Example \ref{ex1.3}. If $n\geq 2$,
\begin{align*}
P_\cap(T_n)&=X(X-1)\ldots (X-n+1),&
P_\subset(T_n)&=X^n-X.
\end{align*}
Here are examples of $P_{\subset,\cap}(T_n)$:
\begin{align*}
P_{\subset,\cap}(T_1)&=X\\
&=H_1(X),\\
P_{\subset,\cap}(T_2)&=X(X-1)\\
&=2H_2(X),\\
P_{\subset,\cap}(T_3)&=\dfrac{X(X-1)(2X-1)}{2}\\
&=3H_2(X)+6H_3(X),\\
P_{\subset,\cap}(T_4)&= X^2(X - 1)^2\\
&=4H_2(X)+24H_3(X)+24H_4(X),\\
P_{\subset,\cap}(T_5)&=\dfrac{X(X-1)(2X-1)(3X^2-3X-1)}{6}\\
&=5H_2(X)+70H_3(X)+180H_4(X)+120H_5(X),\\
P_{\subset,\cap}(T_6)&=\dfrac{X^2(X-1)^2(2X^2-2X-1)}{2}\\
&=6H_2(X)+180H_3(X)+900H_4(X)+1440H_5(X)+720H_6(X),\\
P_{\subset,\cap}(T_7)&=\dfrac{X(X-1)(2X-1)(3X^4 -6 X^3 +3X+1)}{6}\\
&=7H_2(X)+434H_3(X)+3780H_4(X)+10920H_5(X)+12600H_6(X)+5040H_7(X).
\end{align*}
The coefficients of $H_k(X)$ in $P_{\subset,\cap}(T_n)$ are given by Entry A282507 of the OEIS \cite{Sloane}.
\end{example}

\subsection{Homogeneous polynomial invariants}

In all this paragraph, we fix $\leftthreetimes \in \{\subset,\cap\}$. 

\begin{prop}
The following map is a bialgebra map from $(\calF[\bfH],m,\Delta^{(\leftthreetimes)})$ to $(\K[X],m,\Delta)$:
\[P_0:\left\{\begin{array}{rcl}
\calF[\bfH]&\longrightarrow&\K[X]\\
G&\longmapsto&X^{|V(G)|}.
\end{array}\right.\]
\end{prop}

\begin{proof}
With the help of \cite[Propositions 3.10 and 5.2]{Foissy40}, let us define a homogeneous morphism 
$P_0:\calF[\bfH]\longrightarrow \K[X]$ with the help of the element $\mu\in \calF[\bfH]_1^*$ defined by
\[\mu(\tun)=1,\] 
where $\tun$ is the unique hypergraph with one vertex. For any nonempty hypergraph $G$ with $n$ vertices,
\begin{align*}
P_0(G)&=\sum_{k=1}^\infty \mu^k\circ \left(\Delta^{(\leftthreetimes)}\right)^{(k-1)}(G) \frac{X^k}{k!}\\
&=\sum_{k=1}^\infty\sum_{V(G)=I_1\sqcup \ldots \sqcup I_k} \mu(G_{\mid_\leftthreetimes I_1})\ldots  
\mu(G_{\mid_\leftthreetimes I_k})\frac{X^k}{k!}\\
&=\sum_{k=1}^\infty\sum_{\substack{V(G)=I_1\sqcup \ldots \sqcup I_k\\ |I_1|=\ldots=|I_k|=1}}\frac{X^k}{k!}\\
&=n!\frac{X^n}{n!}+0\\
&=X^n. \qedhere
\end{align*}
\end{proof}

\begin{remark}
The map $P_0$ is also a Hopf algebra morphism from $(\calF[\bfH],m,\Delta^{(\subset,\cap)})$ and from 
$(\calF[\bfH],m,\Delta^{(\cap,\subset)})$ to $(\K[X],m,\Delta)$.
\end{remark}

We denote by $\leftsquigarrow_\leftthreetimes$  the action of the monoid $\Char(\calF[\bfH])$ of characters of 
$(\calF[\bfH],m,\delta^{(\leftthreetimes)})$ on
the set of Hopf algebra morphisms from $(\calF[\bfH],m\Delta^{(\leftthreetimes)})$ to $(\K[X],m,\Delta)$
induced by $\delta^{(\leftthreetimes)}$, as defined in \cite{Foissy40}: for any $\lambda \in \Char(\calF[\bfH])$,
for any Hopf algebra morphism $\phi:(\calF[\bfH],m,\Delta^{(\leftthreetimes)})\longrightarrow (\K[X],m,\Delta)$,
\begin{align*}
\phi\leftsquigarrow_\leftthreetimes \lambda&=(\phi\otimes \lambda)\circ \delta^{(\leftthreetimes)}.
\end{align*}
Let $\lambda_0$ be the character $\epsilon_\delta\circ P_0$ of $\calF[\bfH]$: for any hypergraph $H$,
\[\lambda_0(H)=P_0(H)(1)=1.\]
By  \cite[Corollary 3.11]{Foissy40}, 
\[P_0=P_\leftthreetimes\leftsquigarrow_\leftthreetimes\lambda_0.\]
In order to "reverse" this formula, let us study the inverses of characters of $\calF[\bfH]$. 

\begin{prop} \label{prop2.4}
We denote by $\star_\leftthreetimes$ the convolution induced by $\delta^{(\leftthreetimes)}$ on the set of characters
of $\calF[\bfH]$. Let $\zeta$ be a character of $\calF[\bfH]$.
\begin{enumerate}
\item Then $\zeta$ is invertible for $\star$ if, and only if, $\zeta(\tun)\neq 0$.
\item If $\zeta(\tun)=\pm 1$ and for any hypergraph $G$, $\zeta(G)\in \Z$,
then for any hypergraph $G$, $\zeta^{\star_\leftthreetimes-1}(G)\in \Z$.
\end{enumerate} 
\end{prop}

\begin{proof} 1.  For any hypergraph $G$, let us denote by $\cc(G)$ the number of its connected components. 
We put $\deg(G)=|V(G)|-\cc(G)$. For any hypergraph $G$, for any $\sim\in \eq_c[G]$,
\begin{align*}
\cc(G/\sim)&=\cc(G),&|V(G/\sim)|&=\cl(\sim),\\
\cc(G\mid \sim)&=\cl(\sim),&|V(G\mid \sim)|&=|V(G)|,
\end{align*}
where $\cl(\sim)$ is the number of equivalence classes of $\sim$. Therefore, 
$\deg$ induces a graduation of the bialgebra $(\calF[\bfH],m,\delta^{(\leftthreetimes)})$. 
The result is then a direct consequence of \cite[Lemma 3.9]{Foissy43}, where the family of group-like elements is reduced to $\tun$. \\

2. We proceed by induction on $\deg(G)$. If $\deg(G)=0$, then $\delta^{(\leftthreetimes)}(G)=G\otimes G$
and we deduce that 
\[\zeta^{\star_\leftthreetimes-1}(G)=\dfrac{1}{\zeta(G)}=\pm 1.\]
Let us assume that the result is satisfied for any graph $H$ of degree $<\deg(G)$. 
Then
\[\delta^{(\leftthreetimes)}(G)=G\otimes \tun^{|V(G)|}+\tun^{|\cc(G)|}\otimes G+ \sum G'\otimes G'',\]
with $G'$ and $G''$ are hypergraphs with $\deg(G)',\deg(G'')<n$. 
As $\deg(G)>0$, $G$ has at least one edge, and $\epsilon_\delta(G)=0$. Therefore, we put
\begin{align*}
\zeta^{\star_\leftthreetimes-1}(G)&=-\frac{1}{\zeta(\tun)^{|V(G)|}}\left(\frac{1}{\zeta(\tun)^{\cc(G)}}\zeta(G)+\sum 
\zeta^{\star_\leftthreetimes-1}(G')\zeta(G'')\right)\in \Z,
\end{align*}
as $\zeta(\tun)=\pm 1$ and $\zeta(G'')$, $\zeta^{\star_\leftthreetimes-1}(G')\in \Z$. 
\end{proof}

This can be applied to $\lambda_0$:

\begin{defi}
We denote by $\lambda_\leftthreetimes$ the inverse of $\lambda_0$ for the convolution product$\star_\leftthreetimes$
 associated to $\delta^{(\leftthreetimes)}$. It exists, and for any hypergraph $G$, $\lambda_\leftthreetimes(G)\in \Z$. 
\end{defi}

\begin{prop} \label{prop2.5}
Let $\leftthreetimes \in \{\subset,\cap\}$. For any hypergraph $G$,  $P_\leftthreetimes(G)\in \Z[X]$,
and is a unitary polynomial of degree $|V(G)|$.
Moreover,  the opposite of the  coefficient of $X^{|V(G)|-1}$ in $P_\leftthreetimes(G)$ is:
\begin{itemize}
\item the number of edges of $G$ of cardinality 2 if $\leftthreetimes=\subset$. 
\item $\displaystyle \sum_{e\in E^+(G)} \binom{|e|}{2}$ if $\leftthreetimes=\cap$. 
\end{itemize}
Moreover, for any hypergraph $G$,
\begin{align}
\label{eq1}
P_\leftthreetimes(G)=\sum_{\sim\in \eq_\leftthreetimes[G]} \lambda_\leftthreetimes(G\mid\sim) X^{\cl(\sim)}.
\end{align}
\end{prop}

\begin{proof} 
By Proposition \ref{prop2.4},
\[P_\leftthreetimes=P_0\leftsquigarrow_\leftthreetimes \lambda_\leftthreetimes.\]
This gives (\ref{eq1}).  For any hypergraph $H$,
$\lambda_\leftthreetimes(H)\in \Z$, which leads to the conclusion that the coefficients of $P_\leftthreetimes(G)$ 
are integers. Moreover, the degree of $\phi_\leftthreetimes(G)$ is smaller that $|V(G)|$. The unique $\sim$ contributing
with a term of degree $|V(G)|$ is the equality of $V(G)$, for which $G(\mid \sim)=\tun^{|V(G)|}$, so
$\lambda_\leftthreetimes(G\mid_\leftthreetimes\sim)=1$: $P$ is unitary of degree $|V(G)|$. 

The equivalences $\sim \in \eq_\leftthreetimes[G]$ contributing with a term $X^{|V(G)|-1}$ have exactly one 
class $\{x,y\}$ of cardinality $2$, the other ones being singleton. The connectedness condition implies that
$G_{\mid_\leftthreetimes \{x,y\}}$ should be the graph $\tdeux$. For such an equivalence $\sim$,
\[\lambda_\leftthreetimes(G\mid_\leftthreetimes \sim)=
\lambda_\leftthreetimes(\tdeux \tun^{|V(G)|-2})=\lambda_\leftthreetimes(\tdeux)=-1.\]
Consequently, the coefficient of $X^{|V(G)|-1}$ is the opposite of the number of such equivalences $\sim$,
that is to say the number of pairs $\{x,y\}$ of $G$ such that $G_{\mid_\leftthreetimes \{x,y\}}=\tdeux$.
This leads directly to the conclusion. 
\end{proof}

See Proposition \ref{prop2.19} for more results on the coefficients of $P_\leftthreetimes(G)$.

\begin{remark}
In general, $P_{\subset,\cap}(H)\notin \Z[X]$, For example,
\[P_{\subset,\cap}(T_3)=X^3-\frac{3}{2}X^2+\frac{1}{2}X.\]
\end{remark}

\begin{prop} \label{prop2.6}
There is no coproduct $\delta^{(\subset,\cap)}$ on $\calF[\bfH]$ such that:
\begin{enumerate}
\item $(\calF[\bfH],m,\Delta^{(\subset,\cap)},\delta^{(\subset,\cap)})$ is a double bialgebra.
\item The counit of $\delta^{(\subset,\cap)}$ is $\epsilon_\delta$.
\item The character $\lambda_0$ is invertible for the convolution $\star$ associated to $\delta^{(\subset,\cap)}$ and
for any hypergraph $G$, $\lambda_0^{\star-1}(G)\in \Z$.
\item For any hypergraph $G$, we can write 
\[\delta^{(\subset,\cap)}(G)=\sum_{G_1,G_2\: \mbox{\scriptsize hypergraphs}} a_{G_1,G_2}(G) G_1\otimes G_2,\]
with $a_{G_1,G_2}(G)\in \Z$ for any $G_1,G_2$.
\end{enumerate}
\end{prop}

\begin{proof} Let us assume that such a $\delta^{(\subset,\cap)}$ exists. 
The unique double bialgebra morphism $\phi$ from $(\calF[\bfH],m,\Delta^{(\subset,\cap)},\delta^{(\subset,\cap)})$
to $(\K[X],m,\Delta,\delta)$ is the unique bialgebra morphism $\phi$ from $(\calF[\bfH],m,\Delta^{(\subset,\cap)})$
to $(\K[X],m,\Delta)$ such that $\epsilon_\delta\circ \phi=\epsilon_\delta$: it is $P_{(\subset,\cap)}$. 
The morphism $P_0$ is also a bialgebra morphism from $(\calF[\bfH],m,\Delta^{(\subset,\cap)})$
to $(\K[X],m,\Delta)$. Denoting by $\leftsquigarrow$ the action induced by $\delta^{(\subset,\cap)}$,
\[P_0=P_{(\subset,\cap)}\leftsquigarrow \lambda_0.\]
As $\lambda_0$ is invertible, $P_{(\subset,\cap)}=P_0\leftsquigarrow \lambda_0^{\star-1}$. Therefore, for any hypergraph $G$,
\[P_{(\subset,\cap)}(G)=\sum_{G_1,G_2\: \mbox{\scriptsize hypergraphs}} a_{G_1,G_2}(G) 
\lambda_0^{\star-1}(G_2)X^{|V(G_1)|}\in \Z[X],\]
which is not the case for $G=T_3$.
\end{proof}

\begin{cor}\label{cor2.7}
There is no coproduct $\delta$ making $(\calF[\bfH],m,\Delta^{(\subset,\cap)},\delta)$ a double bialgebra, of the form
\[\delta(G)=\sum_{\sim \in \eq'[G]} G/\sim\otimes G_{\mid_\leftthreetimes \sim},\]
where $\eq'[G]$ is a set of equivalences on $V(G)$ and $\leftthreetimes\in \{\cap,\subset\}$. 
\end{cor}

\begin{proof} Indeed, for such a coproduct:
\begin{itemize}
\item The compatibility with the product implies that if $\sim\in \eq'[G]$, then any class of $\sim$ is included into a single
connected component of $G$.
\item The existence of the counit implies then that the equality of $V(G)$ belongs to $\eq'[G]$, as well as the one which 
classes are the connected components of $G$. Consequently, the counit is $\epsilon_\delta$.
\item Adapting the proof of Proposition \ref{prop2.4}, we obtain the condition on $\lambda_0^{\star-1}$. \qedhere
\end{itemize}\end{proof}

\subsection{Acyclic orientations}

If $G$ is a graph, Stanley's theorem \cite{Stanley1973} gives that
\[P_{chr}(G)(-1)=(-1)^{|V(G)|}\sharp\{\mbox{acyclic orientations of $G$}\}.\]
We here extend this result to $P_\subset(G)$ and $P_\cap(G)$ when $G$ are hypergraphs. \\

\begin{notation}
Let $X$ be a set. Recall that a quasi-order on $X$ is a transitive and reflexive relation $\leq$ on $X$. 
It is called total if for any $x,y\in X$, $x\leq y$ or $y\leq x$ (note that $x\leq y$, $y\leq x$ and $x\neq y$ may happen). 
If $\leq$ is a quasi-order on $X$, we define an equivalence on $X$ by
\begin{align*}
&\forall x,y\in X,&x\sim y&\Longleftrightarrow x\leq y\mbox{ and }y\leq x.
\end{align*}
The number of classes of $\sim$ is denoted by $\cl(\leq)$. The set $X/\sim$ is given an order by
\begin{align*}
&\forall \overline{x},\overline{y}\in X/\sim,&\overline{x}\preceq \overline{y}&\Longleftrightarrow x\leq y.
\end{align*}
\end{notation}

\begin{defi}\label{defi2.8}
Let $G$ be a hypergraph. 
\begin{enumerate}
\item An acyclic orientation of $G$ is a quasi-order $\leq$ on $V(G)$ such that:
\begin{itemize}
\item For any $e\in E^+(G)$, $\leq_{\mid e}$ is a total nontrivial quasi-order on $e$.
\item For any $x,y\in V(G)$ such that $x<y$, there exists a path $(x_0,\ldots,x_k)$ in $G$
with $x=x_0$, $y=x_k$ and $x_0<\ldots<x_k$. 
\item For any $x,y\in V(G)$, if $x\leq y$ and $y\leq x$, then $x,y$ belong to a same edge of $G$. 
\end{itemize} 
\item Let $\leq$ be an acyclic orientation of $G$. 
\begin{itemize}
\item We shall say that $\leq$ is total if for any edge $e$, $\leq_{\mid e}$ is an order (hence, a total order).
\item We shall say that $\leq$ is 1-max if for any edge $e$, the maximal class of $\leq_{\mid e}$ is a singleton.
\end{itemize}
\end{enumerate}
\end{defi}

\begin{remark}
Let $G$ be a graph, considered as a hypergraph and let $\leq$ be an acyclic orientation of $G$. 
By the first point, for any edge $\{x,y\}$ of $G$, $x<y$ or $y<x$: we obtain an orientation of $G$
by orienting any edge $e$ according to $<$. As the vertices in an oriented  path of $G$ are strictly increasing
according to $<$, there is no cycle in this orientation: we recover an acyclic orientation of $G$ in its usual sense.
Conversely, if $G'$ is an acyclic orientation of $G$ in the usual sense, we define a partial order on $V(G)$
by $x\leq y$ if there exists an oriented path from $x$ to $y$ in $G'$. It is not difficult to see that this is an acyclic
orientation of the hypergraph $G$. 
Hence, acyclic orientations of graphs $G$ (seen as hypergraphs) are acyclic orientations in the usual sense.
\end{remark}

\begin{lemma}
For any hypergraph $G$,
for any $\leftthreetimes \in \{\subset,\cap\}$,
\[P_\leftthreetimes(G)(-1)=\sum_{n=1}^{|V(G)|} \sum_{\substack{V(G)=I_1\sqcup \ldots \sqcup I_n\\I_1,\ldots,I_n\neq \emptyset,\\
\forall p\in [n],\: E^+(G_{\mid_\leftthreetimes I_p})=\emptyset}} (-1)^n,\]
and
\[P_{(\subset,\cap)}(G)(-1)=\sum_{n=1}^{|V(G)|} \sum_{\substack{V(G)=I_1\sqcup \ldots \sqcup I_n\\I_1,\ldots,I_n\neq \emptyset,\\
\forall p\in [n],\: E^+(G_{\mid^{(p)} I_p})=\emptyset}} (-1)^n.\]
\end{lemma}

\begin{proof}
Recall that $P_\leftthreetimes:(\calF[\bfH],m,\Delta^{(\leftthreetimes)})\longrightarrow (\K[X],m,\Delta)$
is a bialgebra morphism, so is a Hopf algebra morphism. Let $G$ be a hypergraph.
\begin{align*}
P_\leftthreetimes(G)(-X)&=S\circ P_\leftthreetimes(G)=P_\leftthreetimes \circ S_\leftthreetimes(G),
\end{align*}
where $S_\leftthreetimes(G)$ is the antipode of the Hopf algebra $(\calF[\bfH],m,\Delta^{(\leftthreetimes)})$
and $S$ the antipode of $(\K[X],m,\Delta)$. Moreover, by Takeuchi's formula \cite{Takeuchi1971},
\begin{align*}
S_\leftthreetimes(G)&=\sum_{n=1}^{|V(G)|} \sum_{\substack{V(G)=I_1\sqcup \ldots \sqcup I_n\\I_1,\ldots,I_n\neq \emptyset}}
(-1)^n G_{\mid_\leftthreetimes I_1}\ldots G_{\mid_\leftthreetimes I_n}.
\end{align*}
Hence,
\begin{align*}
P_\leftthreetimes(G)(-1)&=\sum_{n=1}^{|V(G)|} \sum_{\substack{V(G)=I_1\sqcup \ldots \sqcup I_n\\I_1,\ldots,I_n\neq \emptyset}}
(-1)^n P_\leftthreetimes(G_{\mid_\leftthreetimes I_1})(1)\ldots P_\leftthreetimes(G_{\mid_\leftthreetimes I_n})(1)\\
&=\sum_{n=1}^{|V(G)|} \sum_{\substack{V(G)=I_1\sqcup \ldots \sqcup I_n\\I_1,\ldots,I_n\neq \emptyset}}
(-1)^n \epsilon_\delta(G_{\mid_\leftthreetimes I_1})\ldots \epsilon_\delta(G_{\mid_\leftthreetimes I_n})\\
&=\sum_{n=1}^{|V(G)|} \sum_{\substack{V(G)=I_1\sqcup \ldots \sqcup I_n\\I_1,\ldots,I_n\neq \emptyset,\\
\forall p\in [n],\: E^+(G_{\mid_\leftthreetimes I_p})=\emptyset}} (-1)^n.
\end{align*}
The proof is similar for $P_{(\subset,\cap)}$. 
\end{proof}

\begin{theo}\label{theo2.10}
Let $G$ be a hypergraph. 
\begin{align*}
P_\subset(G)(-1)&=\sum_{\mbox{\scriptsize $\leq$ acyclic orientation of $G$}} (-1)^{\cl(\leq)},\\
P_\cap(G)(-1)&=(-1)^{|V(G)|}|\{\mbox{total acyclic orientations of $G$}\}|,\\
P_{\subset,\cap}(G)(-1)&=\sum_{\mbox{\scriptsize $\leq$ 1-max acyclic orientation of $G$}} (-1)^{\cl(\leq)}.
\end{align*}
\end{theo}

\begin{proof}
Let $\leq$ be an acyclic orientation of the hypergraph $G$ and let $\leq'$ be a linear extension of $\leq$:
$\leq'$ is a total quasi-order on $V(G)$ such that 
\begin{align*}
&\forall x,y\in V(G),&x\leq y&\Longrightarrow x\leq' y,\\
&&x\leq y\mbox{ and }y\leq x&\Longleftrightarrow x\leq' y\mbox{ and }y\leq' x.
\end{align*}
Let $I_1,\ldots,I_k$ be the classes of $\sim'$, indexed in such a way that for any $(x_1,\ldots,x_k)\in I_1\times \ldots \times I_k$,
$x_1\leq'\ldots \leq' x_k$. For any nontrivial edge $e\in E(G)$, $\leq_{\mid e}$ is a nontrivial total quasiorder,
so is equal to $\leq'_{\mid e}$ which in turn is nontrivial. As a consequence, no nontrivial edge is included
in a single class of $\sim'$: for any $p\in [k]$, $E^+(G_{\mid_\subset I_k})=\emptyset$. 

If $\leq$ is a quasi-order on a set $X$, a linear extension of $\leq$ is a total quasi-order $\leq'$ on the same set $X$,
such that 
\begin{align*}
&\forall x,y\in X,&
x\leq y\mbox{ and }y\leq x&\Longleftrightarrow x\leq' y\mbox{ and }y\leq' x,\\
&&x\leq y&\Longrightarrow x\leq' y.
\end{align*}
We put
\begin{align*}
A&=\{(\leq,\leq')\mid\mbox{$\leq$ acyclic orientation of $G$,\: $\leq'$ linear extension of $\leq$}\},\\
B&=\{(I_1,\ldots,I_k)\mid V(G)=I_1\sqcup \ldots \sqcup I_n,\:I_1,\ldots,I_n\neq \emptyset,\:
\forall p\in [n],\: E^+(G_{\mid_\subset I_p})=\emptyset\},
\end{align*}
and, with the preceding notations, we obtain a map
\[\iota:\left\{\begin{array}{rcl}
A&\longrightarrow&B\\
(\leq,\leq')&\longrightarrow&(I_1,\ldots,I_k).
\end{array}\right.\]

Let us prove that $\iota$ is injective. If $\iota(\leq,\leq')=\iota(\preceq,\preceq')$, then the classes of 
$\leq'$ and $\preceq'$ are the same, and in the same order:  $\leq'=\preceq'$. Let us assume that $x <y$.
As $\leq$ is an acyclic orientation of $G$, there exists a path $(x=x_0,\ldots,x_k=y)$ in $G$,
with $x_0<\ldots <x_k$. Then $x_0<'\ldots <' x_k$, so $x_0\prec'\ldots \prec'x_k$.
Let $p\in [k]$. $x_{p-1}$ and $x_p$ are in the same edge $e\in E(G)$. As $\preceq_{\mid e}$ is a total quasi-order,
$\preceq_{\mid e}=\preceq_{\mid e}'$, so $x_{p-1}\prec x_p$. By transitivity, $x\prec y$. By symmetry,
$\prec=<$, so $\preceq=\leq$.\\

Let us prove that $\iota$ is surjective. Let $(I_1,\ldots, I_n)\in B$. We define a total quasi-order $\leq'$ on $V(G)$ 
by $x\leq' y$ if $x\in I_p$ and $y\in I_q$, with $p\leqslant q$. We then define a partial quasi-order $\leq$ on $V(G)$
by $x\leq y$ if there exists a path $(x=x_0,\ldots,x_k=y)$ in $G$ with for any $p\in [k]$, $x_{p-1}<'x_p$.
Then $\leq'$ is a linear extension of $\leq$, and it is not difficult to prove that $\leq$ is an acyclic orientation of $G$. 
Moreover, $\iota(\leq,\leq')=(I_1,\ldots,I_k)$. \\

Therefore,
\begin{align*}
P_\subset(G)(-1)&=\sum_{(I_1,\ldots,I_n)\in B} (-1)^k=\sum_{(\leq,\leq')\in A}(-1)^{\cl(\leq)}\\
&=\sum_{\mbox{\scriptsize $\preceq$ acyclic orientation of $G$}} 
\left(\sum_{\mbox{\scriptsize $\preceq'$ linear extension of $\preceq$}}(-1)^{|V(G)/\sim|}\right).
\end{align*}
Let $\preceq$ be the partial order on $V(G)/\sim$ induced by $\leq$ and $\mathrm{Hasse}(\preceq)$ its Hasse graph. Then,
by the duality principle \cite[Corollary 4.7]{Foissy43}, for any acyclic orientation $\leq$ of $G$,
\begin{align*}
\sum_{\mbox{\scriptsize $\preceq'$ linear extension of $\preceq$}}(-1)^{|V(G)/\sim|}
&= \Ehr_{Str}(\mathrm{Hasse}(\preceq))(-1)= (-1)^{\cl(\leq)},
\end{align*}
where $\Ehr_{Str}$ is the strict Ehrhart polynomial \cite[Proposition 4.4]{Foissy43}. Hence, 
\begin{align*}
P_\subset(G)(-1)&=\sum_{\mbox{\scriptsize $\leq$ acyclic orientation of $G$}} (-1)^{\cl(\leq)}.
\end{align*}

Let us now consider $P_\cap$. We put 
\[B_\cap=\{(I_1,\ldots,I_n)\mid V(G)=I_1\sqcup \ldots \sqcup I_n,\:I_1,\ldots,I_n\neq \emptyset,\:
\forall p\in [n],\: E^+(G_{\mid_\cap I_p})=\emptyset\}.\]
If $(I_1,\ldots,I_k) \in B_\cap$, then for any $I$, $E^+(G_{\mid_\cap I_p})\subset E^+(G_{\mid_\subset I_p})=\emptyset$,
so $(I_1,\ldots,I_k)\in B$. We proved that $B_\cap\subseteq B$. We put $A_\cap=\iota^{-1}(B_\cap)$. 
If $(\preceq,\preceq') \in A_\cap$ then for any edge $e$ of $G$,
for any class $C$ of $\preceq'$, $C\cap e$ is $\emptyset$ or is a singleton. Therefore, $\preceq_{\mid e}=\preceq'_{\mid  e}$ 
(as $\preceq_{\mid e}$ is total), so $\preceq_{\mid e}$ is a total order: $\preceq$ is a total acyclic orientation of $G$.
Conversely, let $\preceq$ is a total acyclic orientation of $G$ and $\preceq'$ a linear extension of $\preceq$.
If $x\sim'_y$, then $x$ and $y$ belong to a same edge of $G$, and then $x\sim'_{\mid e} y$ and,
as $\sim_{\mid e}=\sim'_{\mid e}$ is a total order, $x=y$ and finally $\preceq'$ is a total order. 
Hence, $I_1,\ldots,I_n$ are singletons. Therefore, obviously $(I_1,\ldots,I_n)\in B_\cap$. 
We obtain
\[P_\cap(G)=\sum_{\mbox{\scriptsize $\leq$ total acyclic orientation of $G$}} (-1)^{\cl(\leq)}.\]
Let $\leq$ be a total acyclic orientation of $G$. If $x\sim y$, then $x$ and $y$ belong to a common edge $e$ of $G$.
As $\leq_{\mid e}$ is a total order, $x=y$, so the classes of $\leq$ are singleton and $\cl(\leq)=|V(G)|$, 
which gives the result. \\

Let us finally consider $P_{\cap,\subset}$. 
\[B_{\cap,\subset}=\{(I_1,\ldots,I_n)\mid V(G)=I_1\sqcup \ldots \sqcup I_n,\:I_1,\ldots,I_n\neq \emptyset,\:
\forall p\in [n],\: E^+(G_{\mid_\cap^{(p)} I_p})=\emptyset\}.\]
If $(I_1,\ldots,I_k) \in B_{\cap,\subset}$, then for any $I$, $E^+(G_{\mid I_p})\subset E^+(G_{\mid^{(i)} I_p})
=\emptyset$,
so $(I_1,\ldots,I_k)\in B$. We proved that $B_{\cap,\subset}\subseteq B$. We put $A_{\cap,\subset}=\iota^{-1}(B_{\cap,\subset})$. 
If $(\preceq,\preceq') \in A_{\cap,\subset}$, then for any $p$, for any edge $e$ included in $I_1\sqcup \ldots \sqcup I_p$,
$e\cap I_p$ is empty or is a singleton. Hence, for any  edge $e$, the maximal class of $e$ (for $\preceq$ or for 
$\preceq'$, as they coincide on $e$), is a singleton, so $\preceq$ is 1-max. Conversely, 
Let $\preceq$ be a 1-max acyclic orientation of $G$ and $\preceq'$ be a linear extension of $\preceq$.
Let $1\leq p\leq n$ and $e$ be a nonempty edge of $G_{\mid_\cap^{(p)} I_p}$. There exists an edge $f$ such that 
$f\subseteq I_1\sqcup \ldots\sqcup I_p$ and $e=f\cap I_p$. As $\preceq$ is 1-max, the maximal class of $f$ is a singleton,
so $f\cap I_p$ is a singleton: we obtain that $e$ is trivial, so $G_{\mid^{(p)} I_p}$ has no non trivial edge.
Therefore, $(\preceq,\preceq')\in A_{\subset,\cap}$. We finally get
\begin{align*}
P_{\subset,\cap}(G)(-1)&=\sum_{\mbox{\scriptsize $\leq$ acyclic 1-max orientation of $G$}} (-1)^{\cl(\leq)}. \qedhere
\end{align*}\end{proof}

Let us give another interpretation for $P_\cap$.

\begin{notation}
Let $G$ be a hypergraph. We associate to $G$ a graph $\Gamma(G)$, with $V(\Gamma(G))=V(G)$
and $E(\Gamma(G))$ is the set of pairs $\{x,y\}$ such that there exists $e\in E(G)$, $\{x,y\} \subseteq e$.
In particular, if $G$ is a graph, $\Gamma(G)=G$. This defines a species morphism  from $\bfH$ to the species
of simple graphs $\bfG_s$. 
\end{notation}

\begin{prop}\label{prop2.11}
The map $\Gamma:(\bfH,m,\Delta^{(\cap,\cap)},\delta^{(\cap)})\longrightarrow (\bfG_s,m,\Delta,\delta)$
is a morphism of twisted bialgebra with a contraction-extraction coproduct.
Moreover, $P_\cap=P_{chr}\circ \calF[\Gamma]$.
\end{prop}

\begin{proof}
Obviously, $\Gamma$ is an algebra morphism. 
Let $G\in \calH[X]$ be a hypergraph and $I\subset X$. Then $\Gamma(G_{\mid_\cap I})=\Gamma(G)_{\mid I}$.
If $\sim \in \eq[X]$, then $\sim\in \eq_\cap[G]$ if, and only if, $\sim \in \eq_c[\Gamma(G)]$. Moreover, if this holds,
\begin{align*}
\Gamma(G/\sim)&=\Gamma(G)/\sim,&
\Gamma(G\mid_\cap \sim)&=\Gamma(G)\mid_\cap \sim.
\end{align*}
This implies that $\Gamma$ is a coalgebra morphism.
As a consequence, for any nonunitary commutative bialgebra $V$, the map $\calF_V[\Gamma]:\calF_V[\bfH]
\longrightarrow \calF_V[\bfG]$ is a double bialgebra morphism.
In the particular case $V=\K$, by unicity of the unique double bialgebra morphism from $\calF[\bfH]$ to $\K[X]$,
\[P_\cap=P_{chr}\circ \calF[\Gamma]. \qedhere\]
\end{proof}

Therefore, by Stanley's theorem: 

\begin{prop}
For any hypergraph $G$,
\[P_\cap(G)(-1)=(-1)^{|V(G)|}\sharp\{\mbox{acyclic orientations of $\Gamma(G)$}\}.\]
\end{prop}

\begin{remark} For any hypergraph $G$, $P_\cap(G)$ is the chromatic polynomial of a graph. This is generally not the case for $P_\subset(G)$.
By Example \ref{ex2.1}, $P_\subset(T_n)=X^n-X$: if $n\geq 3$, this is not the chromatic polynomial of a graph, as for such a polynomial,
the non-zero coefficients form a connected sequence. Similarly, $P_{\subset,\cap}(G)$ is generally not the chromatic polynomial of a graph,
as they are generally not with integral coefficients. 
\end{remark}

\subsection{Antipodes}

From \cite[Corollary 2.3]{Foissy40}:

\begin{cor}\label{cor2.13}
For $\leftthreetimes \in \{\cap,\subset\}$, let us denote by $S_\leftthreetimes$ the antipode of $(\calF[\bfH],m,\Delta^{(\leftthreetimes,\leftthreetimes)})$. For any hypergraph $G$,
\begin{align*}
S_\subset(G)&=\sum_{\sim\in \eq_\subset[G]}\left(\sum_{\mbox{\scriptsize $\leq$ acyclic orientation of $G/\sim$}} (-1)^{\cl(\leq)}\right)
G\mid_\subset \sim,\\
S_\subset(G)&=\sum_{\sim\in \eq_\cap[G]}(-1)^{\cl(\sim)}\sharp\{\mbox{$\leq$ acyclic orientation of $G/\sim$}\}
G\mid_\cap \sim.
\end{align*}
\end{cor}

We cannot use the formalism of double bialgebras for the antipode of $(\calF[\bfH],m,\Delta^{(\subset,\cap)})$, 
which we simply denote by $S$. We shall use Takeuchi's formula \cite{Takeuchi1971}: for any nonempty hypergraph $G$,
\begin{align}\label{eq2}
S(G)&=\sum_{k=1}^{|V(G)|}(-1)^k\sum_{\substack{V(G)=I_1\sqcup \ldots \sqcup I_k,\\ I_1,\ldots,I_k\neq \emptyset}}
G_{\mid^{(1)}I_1}\ldots G_{\mid^{(k)}I_k}.
\end{align}
Let us consider the hypergraphs  appearing in this sum. For such a hypergraph $H$,
$V(H)=V(G)$, and the nonempty edges of $H$ are sets of the form $e\cap I_{\theta(e)}$,
where $e$ is a nonempty edge of $G$  and $\theta(e)\in [k]$. This leads to the following definition:

\begin{defi}
Let $G$ be a nonempty hypergraph, $\sim \in \eq[V(G)]$ and $\theta:E(G)\setminus \{\emptyset\}\longrightarrow V(G)/\sim$
be a map such that for any nonempty edge $e$ of $G$, $e\cap \theta(e)\neq \emptyset$.
\begin{enumerate}
\item We denote by $G\mid_\theta \sim$ the graph such that
\begin{align*}
V(G\mid_\theta \sim)&=V(G),&
E(G\mid_\theta \sim)&=\{e\cap \theta(e)\mid e\in E(G)\setminus \{\emptyset\}\}\cup\{\emptyset\}.
\end{align*}
\item We denote by $G/_\theta \sim$ the oriented graph such that $V(G/_\theta \sim)=V(G)/\sim$ 
and with set of arcs defined by the following:
for any edge $e\in E(G)$, for any $\pi \in V(G)/\sim$ such that $\pi\cap e\in \emptyset$ and $\pi\neq \theta(e)$,
there is an arc from $\pi$ to $\theta(e)$ in $G/_\theta \sim$.
\item We shall write that  $(\sim,\theta)\in \eq_{\subset,\cap}[G]$ if the connected components of $G\mid_\theta \sim$
are the classes of $\sim$ and if the oriented graph $G/_\theta \sim$ is acyclic.
\end{enumerate}
\end{defi}

\begin{prop}\label{prop2.15}
For any nonempty hypergraph $G$,
\[S(G)=\sum_{(\sim,\theta)\in \eq_{\subset,\cap}[G]} (-1)^{\cl(\sim)} G\mid_\theta \sim.\]
\end{prop}

\begin{proof}
Let us denote by $\eq'_{\subset,\cap}[G]$ the set of pairs $(\sim,\theta)$ such that the connected components of $G\mid_\theta \sim$
are the classes of $\sim$. Rewriting (\ref{eq2}), we obtain that
\[S(G)=\sum_{(\sim,\theta)\in \eq'_{\subset,\cap}[G]} \left(
\sum_{\substack{V(G)=I_1\sqcup \ldots \sqcup I_k,\\ I_1,\ldots,I_k\neq \emptyset,\\
G_{\mid^{(1)}_{I_1}}\ldots G_{\mid^{(k)}_{I_k}}=G\mid_\theta \sim}}(-1)^k\right) G\mid_\theta \sim.\]
For any $(\sim,\theta) \in \eq'_{\subset,\cap}[G]$, we put
\[P(\sim,\theta)(X)=\sum_{\substack{V(G)=I_1\sqcup \ldots \sqcup I_k,\\ I_1,\ldots,I_k\neq \emptyset,\\
G_{\mid^{(1)}_{I_1}}\ldots G_{\mid^{(k)}_{I_k}}=G\mid_\theta \sim}} H_k(X) \in \K[X], \]
in such a way that (\ref{eq2}) is rewritten as
\[S(G)=\sum_{(\sim,\theta)\in \eq'_{\subset,\cap}[G]}  P(\sim,\theta)(-1)G\mid_\theta \sim.\]
By definition, for any $N\in \N$, $P(G,\sim)(N)$ is the number of maps $f:V(G)\longrightarrow [N]$ such that
\[G_{\mid^{(1)}_{f^{-1}(1)}}\ldots G_{\mid^{(N)}_{f^{-1}(N)}}=G\mid_\theta \sim.\]
We denote by $\calA_N$ the set of such maps $f$ and by $\calB_N$ the set of maps $g:V(G)/\sim \longrightarrow [N]$
such that if $(\pi_1,\pi_2)$ is an arc of $G/_\theta \sim$, then $g(\pi_1)<g(\pi_2)$.
 
Let us now define a bijection between $\calA_N$ and $\calB_N$. Let $f\in \calA_N$. If $x\sim y$, then by definition of 
$\eq_{\subset,\cap}[G]$, $x$ and $y$ are in the same connected component of $G\mid_\theta \sim$, so they necessarily
belong to the same $f^{-1}(i)$ and finally $f(x)=f(y)$. Therefore, $f$ induces a map $\overline{f}:V(G)/\sim\longrightarrow [N]$
such that for any $x\in V(G)$, $\overline{f}(\overline{x})=f(x)$. Let us prove that $\overline{f}\in \calB_N$. 
If $(\pi_1,\pi_2)$ is an arc of $G/_\theta \sim$, there exists an edge $e$ of $G$, such that $e\cap \pi_1\neq \emptyset$,
$e\cap \pi_2\neq \emptyset$, $\pi_1\neq \pi_2=\theta(e)$. As $e\cap \pi_2$ is an edge of $G\mid_\theta \sim=
G_{\mid^{(1)}_{f^{-1}(1)}}\ldots G_{\mid^{(N)}_{f^{-1}(N)}}$, necessarily 
\[\overline{f}(\pi_2)=\max f_{\mid e},\]
and $\overline{f}(\pi_1)<\overline{f}(\pi_2)$, so $\overline{f}\in \calB_N$. We have defined a map
\[\left\{\begin{array}{rcl}
\calA_N&\longrightarrow&\calB_N\\
f&\longmapsto&\overline{f}.
\end{array}\right.\]
It is obviously injective. Let $\overline{f}\in \calB_N$ and let $f:V(G)\longrightarrow [N]$ be the unique map such that 
for any $x\in V(G)$, $f(x)=\overline{f}(\overline{x})$. Let $e$ be a nonempty edge of $G$. By construction of
$G/_\theta \sim$, the maximum of $f$ over $e$ is obtained on $e\cap \theta(e)$, so the contribution of $e$ to the edges of 
$G_{\mid^{(1)}_{f^{-1}(1)}}\ldots G_{\mid^{(N)}_{f^{-1}(N)}}$ is $e\cap \theta(e)$: we obtain that 
$G_{\mid^{(1)}_{f^{-1}(1)}}\ldots G_{\mid^{(N)}_{f^{-1}(N)}}=G\mid_\theta \sim$. As a conclusion,
$f\in \calA_N$ and $\calA_N$ and $\calB_N$ are in bijection.

As a conclusion, $P(\sim,\theta)(X)$ is the strict Ehrhart polynomial $\Ehr_{str}(G/_\theta\sim)$ of the oriented graph $G\mid_\theta \sim$.
If this oriented graph is acyclic, by the duality principle \cite[Corollary 4.7]{Foissy43}, then 
\[P(\sim,\theta)(-X)=(-1)^{|V(G)/\sim|} \Ehr(G/_\theta \sim)(X),\]
where $\Ehr(G/_\theta \sim)$ is the Ehrhart polynomial of $G/_\theta \sim$. In particular, $\Ehr(G/_\theta \sim)(1)$
is the number of maps $f:V(G)/\sim \longrightarrow [1]$ such that for any arc $(\pi_1,\pi_2)$ of $G/_\theta\sim$,
$f(\pi_1)\leq f(\pi_2)$: this is obviously $1$. As a consequence, if $(\sim,\theta)\in \eq'_{\subset,\cap}[G]$,
then $P(\sim,\theta)(-1)=(-1)^{\cl(\sim)}$. Otherwise, $G/_\theta \sim$ is not acyclic, so $\Ehr_{str}(\sim,\theta)(X)=0$,
which implies that $P(G/_\theta \sim)(-1)=0$. The results immediately follows. \end{proof}

\begin{remark}
As $(\calF[\bfH],m,\Delta^{(\cap,\subset)})$ is the coopposite of $(\calF[\bfH],m,\Delta^{(\subset,\cap)})$, its antipode
is $S^{-1}$. Moreover, as $(\calF[\bfH],m)$ is commutative, $S$ is involutive, so $S^{-1}=S$ and the antipode
of $(\calF[\bfH],m,\Delta^{(\cap,\subset)})$ and $(\calF[\bfH],m,\Delta^{(\subset,\cap)})$ are the same. 
\end{remark}

\subsection{Coefficients of the chromatic polynomials}

\begin{notation}
Let $G$ be a hypergraph. For any $i,j\geq 1$, we denote by $\calN_G(i,j)$
the set of hypergraphs $G'$ of $G$ such that
\begin{align*}
V(G')&=V(G),&E^+(G')&\subset E^+(G),&\cc(G')&=i,&|E^+(G')|&=j.
\end{align*}
We denote by $N_G(i,j)$ the cardinality of $\calN_G(i,j)$. 
\end{notation}

\begin{lemma}\label{lem2.17}
Recall that  $\lambda_\subset$ is the inverse of $\lambda_0$ for the convolution product $\star_\subset$ induced by
$\delta^{(\subset)}$. For any hypergraph $G$,
\[\lambda_\subset(G)=\sum_{j\geq 0} (-1)^j N_G(\cc(G),j).\]
\end{lemma}

\begin{proof}
We define $\mu\in \calF[\bfH]^*$ by 
\[\mu(G)=\sum_{j\geq 0} (-1)^j N_G(\cc(G),j),\]
for any hypergraph $G$. 
Let us prove that for any hypergraph $G$, $\lambda_0\star_\subset \mu(G)=\epsilon_\delta(G)$. 
\begin{align*}
\lambda_0\star_\subset \mu(G)&=\sum_{\sim\in \eq_\subset[G]}\mu(G\mid_\subset \sim)\\
&=\sum_{\sim\in \eq_\subset[G]}\sum_{j\geq 0} (-1)^j N_G(\cl(\sim),j).
\end{align*}
There is an obvious bijection
\[\{F\subset E^+(G)\mid [F|=j\}\longrightarrow \bigsqcup_{\sim\in \eq_\subset[G]}\calN_{G\mid_\subset \sim}(\cl(\sim),j)\]
which sends $F\subset E^+(G)$ to the hypergraph $(V(G),F)$, belonging to $\calN_{G\mid_\subset \sim}(\cl'\sim),j)$
where $\sim$ is the equivalence on $V(G)$ which classes are the connected components of the hypergraph $(V(G),F)$.  Hence,
\begin{align*}
\lambda_0\star_\subset \mu(G)&=\sum_{F\subseteq E^+(G)}
(-1)^{|F|}\\
&=\begin{cases}
1\mbox{ if }|E^+(G)|=\emptyset,\\
0\mbox{ otherwise}
\end{cases}\\
&=\epsilon_\delta(G).\qedhere
\end{align*}
 \end{proof}

Let us deduce the following description of the coefficients of $P_\subset(G)$,
which can be found in \cite{Borowiecki2000,Tomescu98}:

\begin{prop}\label{prop2.19}
For any $i\geq 1$, the coefficient $a_i$ of $X^i$ in $P_\subset(G)$ is
\[a_i=\sum_{j\geq 0} (-1)^j N_G(i,j).\]
\end{prop}

\begin{proof}
 From Proposition \ref{prop2.5},
\begin{align}
\label{eqPsubset}
P_\subset(G)&=\sum_{\sim\in \eq_\subset[G]}\lambda_\subset(G\mid \sim)X^{\cl(\sim)}.
\end{align}
Consequently, combining with Lemma \ref{lem2.17}, for any $i\geq 1$,
\[a_i=\sum_{\substack{\sim \in \eq_\subset[G],\\ \cl(\sim)=i}} \sum_{j\geq 0} (-1)^j N_{G\mid \sim}(i,j).\]
Moreover, there is an obvious bijection
\begin{align*}
\calN_G(i,j)&\longrightarrow\bigsqcup_{\substack{\sim \in \eq_\subset[G],\\ \cl(\sim)=i}}\calN_{G\mid_\subset \sim}(i,j),
\end{align*}
sending $G'$ to itself, seen as an an element of $\calN_{G\mid_\subset \sim}(i,j)$, where $\sim$ is the equivalence which classes are the connected components of $G'$. Consequently,
\[a_i=\sum_{j\geq 0} N_G(i,j). \qedhere\]
\end{proof}

\begin{remark}
If $G$ is a hypergraph with $n$ vertices, then $N_G(n-1,j)=0$ if $j\neq 1$ and $N_G(n-1,1)$ is the 
number of edges of $G$ of cardinality 2. We recover the result of Proposition \ref{prop2.5}.
\end{remark}

\begin{prop}
We define a map $\varpi$ on $\calF[\bfH]$ by the following:
for any hypergraph $G$,
\[\varpi(G)=\sum_{\sim \in \eq_\subset[G]}
\left(\sum_{j\geq 0}(-1)^j N_{G/\sim}(1,j)\right)G\mid_\subset\sim.\]
Then $\varpi$ is the projector on the space $\prim(\calF[\bfH])$ of primitive elements of $\calF[\bfH]$ which vanishes on
$(1)\oplus \ker(\varepsilon)^2$ (eulerian idempotent). 
Consequently, a basis of $\prim(\calF[\bfH])$ is given by
$(\varpi(G))_{\mbox{\scriptsize $G$ connected hypergraph}}$. 
\end{prop}

\begin{proof}
By \cite[Proposition 4.1]{Foissy40}, the infinitesimal character $\ln(\epsilon_\delta)$ is given on any hypergraph $G$ by
\[\ln(\epsilon_\delta)(G)=\sum_{j\geq 0} (-1)^j N_G(1,j).\]
We conclude with \cite[Corollary 4.5]{Foissy40}.
\end{proof}

\subsection{Morphisms to quasishuffle algebras}

We assume in this paragraph that $(V,\cdot,\delta_V)$ is a  nonunitary,  commutative and cocommutative bialgebra.
By \cite[Proposition 3.9]{Foissy41}, $\calF_V[\bfH]$ is a bialgebra over $V$, with the coaction $\rho$ described as follows: 
if $G$ is a $V$-decorated hypergraph with $n$ vertices, we arbitrarily index these vertices and we denote by $G(v_1,\ldots,v_n)$ 
the  hypergraph with for any $i$, the $i$-th vertex of $G$ decorated by $v_i$. Then
\[\rho(G(v_1,\ldots,v_n))=G(v_1',\ldots,v_n')\otimes v_1'' \cdot\ldots \cdot v_n''.\]

\begin{notation}
The map $\pi_V:T(V)\longrightarrow \K[X]$ is defined by
\begin{align*}
&\forall v_1,\ldots,v_n \in V,&\pi_V(v_1\ldots v_n)&=\epsilon_V(v_1)\ldots \epsilon_V(v_n)\dfrac{X(X-1)\ldots (X-n+1)}{n!}.
\end{align*}
It is a double bialgebra morphism. 
\end{notation}

By \cite[Theorem 2.7]{Foissy42}:

\begin{prop}\label{prop2.20}
Let $(V,\cdot,\delta_V)$ be  a commutative, not necessarily unitary bialgebra. 
\begin{enumerate}
\item For any hypergraph $G$, we denote by $\mathrm{VC}_\cap(G)$ the set of surjective maps $f:V(G)\longrightarrow [k]$
such that if $x$ and $y$ are two distinct elements of an edge $e\in E(G)$, then $f(x)\neq f(y)$. 
The unique double bialgebra morphism over $V$ from $(\calF_V[\bfH],m,\Delta^{(\cap,\cap)},\delta^{(\cap)})$
to $(T(V),\squplus,\Delta,\delta)$ sends any $V$ decorated hypergraph $G$ to 
\[\Phi_\cap(G)=\sum_{f\in \mathrm{VC}_\cap[G]} \left(\prod_{f(i)=1}^\cdot v_i\right)
\ldots  \left(\prod_{f(i)=\max(f)}^\cdot v_i\right).\]
Moreover, $P_\cap\circ \Theta_V=\pi_V\circ\Phi_\cap$.
\item For any hypergraph $G$, we denote by $\mathrm{VC}_\subset(G)$ the set of surjective maps $f:V(G)\longrightarrow [k]$
such that if $e$ is a nontrivial edge of $G$, then $f$ takes at least two different values on $e$. 
The unique double bialgebra morphism over $V$ from $(\calF_V[\bfH],m,\Delta^{(\subset,\subset)},\delta^{(\subset)})$
to $(T(V),\squplus,\Delta,\delta)$ sends any $V$ decorated hypergraph $G$ to 
\[\Phi_\subset(G)=\sum_{f\in \mathrm{VC}_\subset[G]} \left(\prod_{f(i)=1}^\cdot v_i\right)
\ldots  \left(\prod_{f(i)=\max(f)}^\cdot v_i\right).\]
Moreover, $P_\subset\circ \Theta_V=\pi_V\circ\Phi_\subset$.
\end{enumerate}
\end{prop}

\begin{proof}
Let $\leftthreetimes\in \{\cap,\subset\}$. 
Both maps $P_\leftthreetimes\circ \Theta_V$ and $\pi_V\circ\Phi_\leftthreetimes$ are double bialgebra morphisms from
$(\calF_V[\bfH],m,\Delta^{(\leftthreetimes)},\delta^{(\leftthreetimes)})$ to $(\K[X],m,\Delta,\delta)$.
By unicity of such a morphism, they are equal. 
\end{proof}

Even without the double bialgebra structure, we can define a Hopf algebra morphism for $\Delta^{(\cap,\subset)}$,
with \cite[Theorem 2.3]{Foissy42}:

\begin{prop}\label{prop2.21}
Let $(V,\cdot)$ be  a commutative, not necessarily unitary algebra. 
 For any hypergraph $G$, we denote by $\mathrm{VC}_{\cap,\subset}(G)$ the set of surjective maps $f:V(G)\longrightarrow [k]$
such that if $e$ is a nontrivial edge of $G$, then $\max\{f(x)\mid x\in e\}$ is obtained in exactly one element of $e$.
The following defines a Hopf algebra morphism  from $(\calF_V[\bfH],m,\Delta^{(\cap,\subset)})$
to $(T(V),\squplus,\Delta)$: for any $V$-decorated  hypergraph $G$,
\[\Phi_{\cap,\subset}(G)=\sum_{f\in \mathrm{VC}_{\cap,\subset}[G]} \left(\prod_{f(i)=1}^\cdot v_i\right)
\ldots  \left(\prod_{f(i)=\max(f)}^\cdot v_i\right).\]
Moreover, $P_{\cap,\subset}\circ \Theta_V=\pi_V\circ\Phi_{\cap,\subset}$.
\end{prop}

\section{Multi-complexes}

\subsection{Definition}

Recall that a multiset is a map 
$X:S\longrightarrow \N\setminus \{0\}$, where $S$ is a set, called the support of $X$ and denoted by $\supp(X)$.
For any $x\in\supp(X)$, $X(x)$ is the multiplicity of $x$ in $X$.  
Multisets are usually seen as "sets with repetitions of elements": for example, the multiset 
\[X:\left\{\begin{array}{rcl}
\{a,b,c\}&\longrightarrow&\N\setminus\{0\}\\
a&\longmapsto&1\\
b&\longmapsto&3\\
c&\longmapsto&2
\end{array}\right.\]
is represented by $\{a,b,b,b,c,c\}$. If $X$ and $Y$ are two multisets, $X\subseteq Y$ if
$\supp(X)\subseteq \supp(Y)$ and for any $x\in \supp(X)$, $X(x)\leq Y(x)$. 
For example, $\{a,a,b,b,b,c\}\subseteq \{a,a,a,b,b,b,c,c,c,d\}$. \\

The notion of multi-complexes is introduced in \cite{Iovanov2022}. Let us give a slightly modified definition, adapted to our setting:

\begin{defi}
A multi-complex is a triple $C=(V(C),E(C),\leq_C)$, where:
\begin{itemize}
\item $V(C)$ is a finite set, called the set of vertices of $C$.
\item $E(C)$ is a multiset of multisets, such that:
\begin{itemize}
\item For any $e\in E(C)$, the support $\supp(e)$ of $e$ is a subset of $V(C)$.
\item For any $x\in V(C)$, $\{x\}$ is an element of the multiset $E(C)$ of multiplicity $1$.
\item $\emptyset$ is an element of the multiset $E(C)$ of multiplicity 1.
\end{itemize}
The elements of $E(C)$ are called the edges of $C$. 
\item $\leq_C$ is a partial order on the multiset $E(C)$ such that:
\begin{itemize}
\item For any $x\in E(C)$, for any $e\in E(C)$, $\{x\}\leq_C e$ if, and only if, $x\in \supp(e)$. 
\item For any $e\in E(C)$, $\emptyset \leq_C e$.
\item For any $e,f\in E(C)$, if $e\leq_C f$, then $e\subset f$. 
\end{itemize}
\end{itemize}
For any finite set $X$, the set of multi-complexes $C$ with $V(C)=X$ is denoted by $\calMC[X]$,
and the vector space generated by $\calMC[X]$ is denoted by $\bfMC[X]$. 
Then $\calMC$ is a set species and $\bfMC$ is a species.
\end{defi}

\begin{example}\label{excomplexe} Here is a multicomplex $C$. 
We put $V(C)=\{a,b,c,d\}$, and
\[E(C)=\left\{\begin{array}{c}
\emptyset,\{a\},\{b\},\{c\},\{d\},\\
\{a,b\}, \{a,c\},\{a,c\},\{b,d\},\{c,d\},\{a,b,c\},\{a,a,c\},\{b,b,d\}
\end{array}\right\},\]
with the partial order given by its Hasse graph:
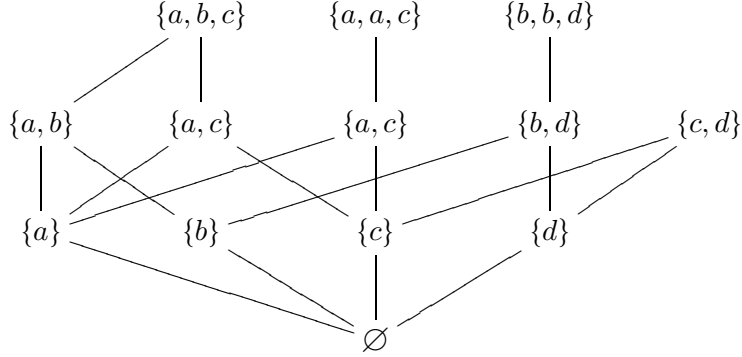
\[\xymatrix{
&\{a,b,c\}&\{a,a,c\}&\{b,b,d\}\\
\{a,b\}\ar@{-}[ru]&\{a,c\}\ar@{-}[u]&\{a,c\}\ar@{-}[u]&\{b,d\}\ar@{-}[u]&\{c,d\}\\
\{a\}\ar@{-}[u]\ar@{-}[ru]\ar@{-}[rru]
&\{b\}\ar@{-}[lu]\ar@{-}[rru]&
\{c\}\ar@{-}[lu]\ar@{-}[u]\ar@{-}[rru]&\{d\ar@{-}[u]\ar@{-}[ru]\}\\
&&\emptyset \ar@{-}[llu] \ar@{-}[lu] \ar@{-}[u] \ar@{-}[ru]}\]
\end{example}

Hypergraphs are multi-complexes, with $\leq$ given by the inclusion; note that in this case, the edges of $C$ are sets, and the multiset of edges is also a set. 
Simplicial complexes and $\Delta$-complexes are also multi-complexes, see \cite{Iovanov2022}.

\subsection{Hopf algebraic structures multi-complexes}

A Hopf algebra of multi-complexes is introduced in \cite{Iovanov2022}. Let us lift this to the twisted level. 
Let $X$ and $Y$ be two disjoint sets. If $C\in \calMC[X]$ and $D\in \calMC[Y]$, the multi-complex $CD\in \calMC[X\sqcup Y]$
is defined by
\begin{align*}
V(CD)&=X\sqcup Y,&E(CD)&=E(C)\sqcup E(Y),
\end{align*}
and for any $e,f\in E(CD)$, 
\[e\leq_{CD}f\mbox{ if ($e,f\in E(C)$ and $e\leq_C f$) or ($e,f\in E(D)$ and $e\leq_D f$)}.\]
This defines a associative, commutative product $m$ on $\bfMC$, which unit is the empty multi-complex.\\

For any finite sets $X\subset Y$ and for any multi-complex $C\in \calMC[Y]$, we define $C_{\mid X}$ by
\begin{align*}
V(C_{\mid X})&=X,&E(C_{\mid X})&=\{e\in E(C)\mid \supp(C)\subset X\},&
\leq_{C_{\mid X}}&=(\leq_C)_{\mid E(C_{\mid X})}.
\end{align*}
This is indeed a multi-complex. We then define a coproduct $\Delta$ on $\bfMC$ by the following:
for any finite sets $X$ and $Y$, for any $C\in \calMC[X\sqcup Y]$,
\[\Delta_{X,Y}(X)=C_{\mid X}\otimes C_{\mid Y}.\]

\begin{prop}
$(\bfMC,m,\Delta)$ is a twisted bialgebra. Moreover, $\calF[\bfMC]$ is the Hopf algebra of multi-complexes of \cite{Iovanov2022}.
\end{prop}

\begin{proof}
Similar to the proof of Proposition \ref{prop1.2}.
\end{proof}

Let us now define an extraction-contraction coproduct on $\bfMC$. \\

Let $C$ be a multi-complex, and let $x,y\in V(C)$. A path from $x$ to $y$ is a sequence $(x_0,\ldots,x_k)$ of vertices of $C$
such that:
\begin{itemize}
\item $x_0=x$ and $x_k=y$.
\item For any $i\in [k]$, there exists $e\in E(C)$ such that $x_{i-1},x_i \in e$.
\end{itemize}
We shall say that $C$ is connected if for any $x,y\in V(C)$, there exists a path from $x$ to $y$ in $C$. \\

Let $X$ be a finite set, $\sim\in \eq[X]$ and $C\in \calMC[X]$. 
\begin{enumerate}
\item We shall say that $\sim \in \eq_c[C]$ if for any 
$\varpi\in X/\sim$, $C_{\mid \varpi}$ is connected.
\item We denote by $X\mid \sim$ the multi-complex defined by
\begin{align*}
V(C\mid \sim)&=V(C),\\
E(C\mid \sim)&=\{e\in E(C)\mid \forall x,y\in \supp(e),\: x\sim y\},\\
\leq_{C\mid \sim}&=(\leq_C)_{\mid E(C\mid \sim)}.
\end{align*}
In other words,
\[C\mid \sim=\prod_{\varpi \in X/\sim} C_{\mid \varpi}.\]
\item We denote by $X/\sim$ the multi-complex defined by
\begin{align*}
V(C/\sim)&=V(V)/\sim,\\
E(C/\sim)&=\{\pi(e)\mid e\in E(C)\},
\end{align*}
where $\pi_\sim:V(C)\longrightarrow V(C)/\sim$ is the canonical surjection. It is noticeable
that $E(C/\sim)$ is a multiset, that is to say we distinguish all the $\pi(e)$, $e\in E(C)$, in $E(C/\sim)$,
except for the trivial edges (which are $\empty$ and the singletons, which remains of multiplicity 1). 
In other terms, if $\overline{e}$ is a multiset of support included in $V(C)/\sim$, its multiplicity in $E(C/\sim)$
is the sum of the multiplicities of the edges $e \in E(C/\sim)$ such that $\pi(e)=\overline{e}$. 
The partial order on $E(C/\sim)$ is defined by 
\[\pi_\sim(e)\leq_{C/\sim} \pi_\sim(f)\Longleftrightarrow e\leq_C f.\]
\end{enumerate}

\begin{example}
Let us consider the multi-complex $C$ of Example \ref{excomplexe} again. 
Let $\sim$ be the equivalence which classes are $\{a,b\}$ and $\{c,d\}$. Because its classes are edges of $C$,
$\sim\in \eq_c[C]$. Moreover, 
$V(C\mid \sim)=\{a,b,c,d\}$, and
\[E(C\mid \sim)=\left\{\begin{array}{c}
\emptyset,\{a\},\{b\},\{c\},\{d\},\\
\{a,b\}, \{b,d\}\end{array}\right\},\]
with the partial order given by its Hasse graph:
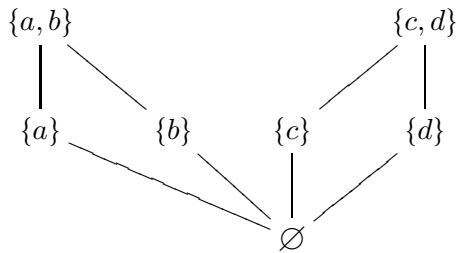
\[\xymatrix{
\{a,b\}&&&\{c,d\}\\
\{a\}\ar@{-}[u]&\{b\}\ar@{-}[lu]&
\{c\}\ar@{-}[ru]&\{d\}\ar@{-}[u]\\
&&\emptyset \ar@{-}[llu] \ar@{-}[lu] \ar@{-}[u] \ar@{-}[ru]
}\]
Moreover, $V(C/\sim)=\{\overline{a},\overline{c}\}$, and
\[E(C)=\left\{\begin{array}{c}
\emptyset,\{\overline{a}\},\{\overline{c}\},\\
\{\overline{a},\overline{a}\}, \{\overline{a},\overline{c}\},\{\overline{a},\overline{c}\},\{\overline{a},\overline{c}\},
\{\overline{c},\overline{c}\},\{\overline{a},\overline{a},\overline{c}\},\{\overline{a},\overline{a},\overline{c}\},
\{\overline{a},\overline{a},\overline{c}\}
\end{array}\right\},\]
with the partial order given by its Hasse graph:
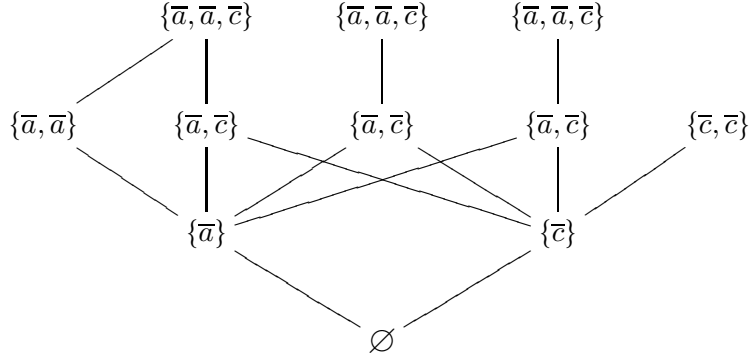
\[\xymatrix{
&\{\overline{a},\overline{a},\overline{c}\}&\{\overline{a},\overline{a},\overline{c}\}&\{\overline{a},\overline{a},\overline{c}\}\\
\{\overline{a},\overline{a}\}\ar@{-}[ru]&\{\overline{a},\overline{c}\}\ar@{-}[u]&\{\overline{a},\overline{c}\}\ar@{-}[u]&\{\overline{a},\overline{c}\}\ar@{-}[u]&\{\overline{c},\overline{c}\}\\
&\{\overline{a}\}\ar@{-}[lu]\ar@{-}[u]\ar@{-}[ru]\ar@{-}[rru]&
&\{\overline{c}\ar@{-}[u]\ar@{-}[ru]\ar@{-}[llu]\ar@{-}[lu]\}\\
&&\emptyset \ar@{-}[lu] \ar@{-}[ru]}\]
\end{example}

\begin{theo}
For any multi-complex $C\in \calMC[X]$ and for any $\sim\in \eq[X]$, we put
\[\delta_\sim(C)=\begin{cases}
C/\sim\otimes \:C\mid\sim\mbox{ if }\sim\in \eq_c[C],\\
0\mbox{ otherwise}.
\end{cases}\]
This defines a contraction-extraction coproduct on $\bfMC$ in  the sense of \cite{Foissy41}, 
compatible with $m$ and $\Delta$. 
\end{theo}

\begin{proof}
Similar to the proof of Theorem \ref{theo1.4}. 
\end{proof}

\section{Link with hypergraphs}

\begin{defi}
Let $C$ be a multi-complex. We define the hypergraph $\kappa(C)$ by
\begin{align*}
V(\kappa(C))&=V(C),&
E(\kappa(C))&=\supp\{\supp(e)\mid e\in E(C)\}.
\end{align*}
In other words, $\kappa(C)$ is obtained from $C$ by forgetting the partial order $\leq_C$ and the multiplicities
in the edges and in $E(C)$. This defines a species morphism $\kappa:\bfMC\longrightarrow \bfH$.
\end{defi}

The following is obtained by direct verifications:

\begin{prop}
$\kappa:(\bfMC,m,\Delta)\longrightarrow (\bfH,m,\Delta^{(\subset)})$ is a twisted bialgebra morphism.
Moreover, it is compatible with the contraction-extraction coproducts $\delta$ and $\delta^{(\subset)}$. 
\end{prop}

As a consequence, the unique double bialgebra morphism from $\calF[\bfMC]$ to $\K[X]$ is
$P_\subset \circ \calF[\kappa]$. \\

From \cite[Corollary 2.3]{Foissy40}:

\begin{cor}
Let us denote by $S$ the antipode of $(\calF[\bfMC],m,\Delta)$.
For any mutli-complex $C$,
\begin{align*}
S(C)&=\sum_{\sim\in \eq_{c}[C]}\left(\sum_{\mbox{\scriptsize $\leq$ acyclic orientation of $\kappa(C/\sim)$}} (-1)^{\cl(\leq)}\right)
C\mid \sim.
\end{align*}
\end{cor}

By \cite[Corollary 4.5]{Foissy40}:

\begin{prop}
We define a map $\varpi$ on $\calF[\bfMC]$ by the following:
for any multi-complex $C$,
\[\varpi(C)=\sum_{\sim \in \eq_c[C]}
\left(\sum_{j\geq 0}(-1)^j N_{\kappa(C/\sim)}(1,j)\right)C\mid\sim.\]
Then $\varpi$ is the projector on the space $\prim(\calF[\bfMC])$ of primitive elements of $\calF[\bfMC]$ which vanishes on
$(1)\oplus \ker(\varepsilon)^2$ (eulerian idempotent). 
Consequently, a basis of $\prim(\calF[\bfMC])$ is given by
$(\varpi(C))_{\mbox{\scriptsize $C$ connected multi-complex}}$. 
\end{prop}

\bibliographystyle{amsplain}
\bibliography{biblio}

\end{document}